\documentclass{amsart}
\usepackage{amsmath,amsthm}
\usepackage{amsfonts,amssymb}

\newtheorem{thm}{Theorem}[section]
\newtheorem{cor}[thm]{Corollary}
\newtheorem{lem}[thm]{Lemma}
\newtheorem{prop}[thm]{Proposition}

\newtheorem{defn}[thm]{Definition}

\theoremstyle{remark}
\newtheorem{rem}{Remark}[section]

\def\ta{\theta}
\def\al{{\alpha}}
\def\be{{\beta}}
\def\da{{\delta}}
\def\ga{{\gamma}}
\def\og{{\omega}}
\def\ld{{\lambda}}
\def\Bl{\Bigl}
\def\Br{\Bigr}
\def\f{\frac}
\def\tf{\tfrac}
\def\vi{\varphi}
\def\({\left(}
\def \){ \right)}
\def\[{\left[}
\def \]{ \right]}
\newcommand\p{\partial}
\def\i{{\textnormal{i}}}
 \def\a{{\alpha}}
 \def\b{{\beta}}
 \def\g{{\gamma}}
 \def\k{{\kappa}}
 \def\t{{\theta}}
 \def\l{{\lambda}}
 \def\d{{\delta}}
 \def\o{{\omega}}
 \def\s{{\sigma}}
 \def\la{{\langle}}
 \def\ra{{\rangle}}
\def\va{\varepsilon}
 \def\ve{{\varepsilon}}

 \def\CD{{\mathcal D}}
 
 \def\CH{{\mathcal H}}

 \def\CO{{\mathcal O}}

 \def\CS{{\mathcal S}}
 
 \def\CV{{\mathcal V}}

 \def\NN{{\mathbb N}}

 \def\RR{{\mathbb R}}
 \def\ZZ{{\mathbb Z}}
        
        \def\proj{\operatorname{proj}}

\def\be{\beta}

\def\sa{\sigma}

\newcommand{\wh}{\widehat}

\begin{document}

\title[projection operators and Ces\`aro means in weighted $L^p$ space]
{Boundedness of projection operators and Ces\`aro means in
weighted $L^p$ space on the unit sphere}
\author{Feng Dai}
\address{Department of Mathematical and Statistical Sciences\\
University of Alberta\\, Edmonton, Alberta T6G 2G1, Canada.}
\email{dfeng@math.ualberta.ca}
\author{Yuan Xu}
\address{Department of Mathematics\\ University of Oregon\\
    Eugene, Oregon 97403-1222.}\email{yuan@math.uoregon.edu}

\date{\today}
\thanks{The first  author  was partially supported  by the NSERC Canada
under grant G121211001. The second author was partially supported by the
NSF under Grant DMS-0604056}

\begin{abstract}
For the weight function $\prod_{i=1}^{d+1}|x_i|^{2\k_i}$ on the
unit sphere, sharp local estimates of the orthogonal projection
operators are obtained   and used to prove the convergence of the
Ces\`aro $(C,\delta)$ means in the weighted $L^p$ space for
$\delta$ below the critical index. Similar results are also proved
for corresponding weight functions on the unit ball and on the
simplex.
\end{abstract}

\maketitle

\section{Introduction}
\setcounter{equation}{0}

For spherical harmonic expansions on the unit sphere
$S^d:=\{(x_1,\cdots, x_{d+1}):(x_1^2+\cdots+x_{d+1}^2)^{\f12}=1\}$
of $\RR^{d+1}$, it is well known that their Ces\'aro $(C,\d)$
means are uniformly bounded in $L^p$ norm for all $1 \le p \le
\infty$ if and only if $\d \ge \frac{d-1}2$ (\cite{BC}). For $\d$
below the critical index $ \frac{d-1}2$, C.  Sogge \cite{So1}
proved a much deeper result that the $(C,\d)$ means are uniformly
bounded on $L^p(S^d)$ if
\begin{equation}\label{sogge}
|\tfrac12-\tfrac1p| \ge \tfrac1{d+1}\quad \hbox{and} \quad
   \delta > \delta(p): = \max\{d | \tfrac1p-\tfrac12 | -\tfrac12, 0\}
\end{equation}
for $d \ge 2$ and, moreover, the condition  $|1/2-1/p| \ge
1/(d+1)$ is not needed in the case of $d =2$. The condition $\d >
\d(p)$ is also known to be necessary (\cite{BC}). Later in
\cite{So2}, Sogge proved that the condition \eqref{sogge} ensures
the boundededness of the Riesz means of eigenfunction expansions
associated to the second order elliptic differential operators on
compact connected $C^\infty$ manifolds of dimension $d$.

 The purpose of the present paper is to establish analogous
results for the Ces\'aro means of orthogonal expansions associated
with the weight function $h_\k^2(x)$, where
 \begin{equation}  \label{weight}
    h_\k(x): = \prod_{i=1}^{d+1} |x_i|^{\k_i},
       \    \  \k:=(\k_1,\cdots,\k_{d+1}),  \   \  \min_{1\leq i\leq
       d+1} \k_i\ge 0,
\end{equation}
 on the unit sphere
$S^{d}$,
 as well as for orthogonal expansions for related weight
functions (see \eqref{weightB} and \eqref{weightT} below) on the
unit ball and on the simplex.

The function $h_\k$ in  (\ref{weight})  is invariant under the
group $\ZZ_2^d$ and it is the simplest example of weight functions
invariant under reflection groups studied first by Dunkl
\cite{D1}. Homogeneous polynomials that are orthogonal with
respect to $h_\k^2$ on $S^d$ are called $h$-harmonics and their
restrictions on $S^d$ are eigenfunctions of a second order
differential-difference operator, which plays the role of the
ordinary Laplacian. For the theory of $h$-harmonics, we refer to
\cite{DX} and the references therein. A brief account of what is
needed in this paper is given in the following section.

The convergence of $h$-harmonic expansions has been studied
recently. In \cite{X97} it was proved that Ces\`aro $(C,\d)$ means
converge uniformly if $\d > |\k| + \frac{d-1}2$, where $|\k| =
\k_1 + \ldots + \k_{d+1}$, and such a result holds for all other
weight functions invariant under reflection groups. In the case of
$h_\k$ in \eqref{weight} the critical index for the $(C,\d)$ means
in the uniform norm turned out to be (\cite{LX})
\begin{equation} \label{critical-index}
   \delta > \sigma_\k: = \tfrac{d-1}{2} + |\k| - \min_{1 \le i \le d+1} \k_i.
\end{equation}
Our main result in this paper,  (see Theorem \ref{thm:Cesaro} in
Section 3), shows that  for $h_\k$ in \eqref{weight}, the $(C,\d)$
means of $h$-harmonic expansions converge in the $L^p(h_\k^2;
S^d)$ norm if
\begin{equation}\label{eq:main}
  |\tfrac12-\tfrac1p| \ge \tfrac1{2 \sigma_\k+2}\quad \hbox{and} \quad
   \delta > \delta_\k(p): = \max\{(2 \sigma_\k +1) | \tfrac1p-\tfrac12 | -\tfrac12, 0\}
\end{equation}
and  that the condition $\d > \d_\k(p)$ is also necessary. Note
that \eqref{eq:main} agrees with \eqref{sogge} when $\k = 0$,
while $\da_\k(p)>\da(p)$ when $\k>0$.

The reason that these sharp results can be established for $h_\k$
in \eqref{weight} lies in an explicit formula for the kernel
$P_n(h_\k^2;\cdot,\cdot)$ of the orthogonal projection operator
(definition in the next section), while no explicit formula for
the kernel is known for other reflection invariant weight
functions. For \eqref{weight}, we have
\begin{align} \label{proj-Kernel}
 P_n(h_\k^2;x,y) = c_\k  \frac{n+\l_\k}{\l_\k} \int_{[-1,1]^{d+1}}
      C_n^{\l_\k}(u(x,y,t)) \prod_{i=1}^{d+1}(1+t_i) (1-t_i^2)^{\k_i-1}dt,
\end{align}
where $C_n^\l$ is the Gegenbauer polynomial of degree $n$,
\begin{equation} \label{lambda}
  \l_\k := \tfrac{d-1}{2}  +|\k|,   \quad \hbox{and} \quad u(x,y,t)=
      x_1y_1t_1+\ldots + x_{d+1}y_{d+1}t_{d+1},
\end{equation}
and $c_\k$ is the normalization constant of the weight function
$\prod_{i=1}^d(1+t_i) (1-t_i^2)^{\k_i-1}$. If some $\kappa_i =0$,
then the formula  holds under the limit relation $$
 \lim_{\lambda \to 0} c_\lambda \int_{-1}^1 g(t)(1+t) (1-t)^{\lambda -1} dt
  = g(1).$$  For the spherical harmonic expansions, this kernel is the
familiar $P_n(x,y):=\frac{n+\l}{\l} C_n^\l(\la x,y\ra)$ with $\l =
\frac{d-1}2$ (cf. \cite{St}), which is also called a zonal
harmonic.

The simple structure of the zonal harmonics means that one can
derive various properties and estimates relatively easily. The
structure of the kernel $P_n(h_\k^2; x,y)$ in \eqref{proj-Kernel}
is far more complicated, making derive information from it more
difficult. There is, however, a deeper reason that the study of
$h$-harmonic expansion is more difficult than that of ordinary
spherical harmonic. The zonal harmonics are invariant under the
rotation group $O(d+1)$ in the sense that $P_n(x,y) = P_n(x g, y g
)$ for all $g \in O(d+1)$, which reflects the fact that the sphere
is a homogeneous space. The $P_n(h_\k^2;x,y)$ in
\eqref{proj-Kernel} is invariant under $\ZZ_2^{d+1}$, a subgroup
of $O(d+1)$, and we are in fact working with a weighted sphere
that has singularity on the largest circles of the coordinate
planes. In particular, we can no longer treat the sphere as a
homogeneous space and many of our estimates of various kernels
have to be local, depending on the location of the points.

The difficulty manifests acutely in the study of the $L^p$
boundedness of Ces\`aro means of $h$-harmonic expansions. For the
ordinary spherical harmonics, the proof of Sogge  \cite{So1}
relies on the sharp asymptotic bounds  for the $(L^p,L^2)$ norms
of the orthogonal projection operators, and  a result of
Bonami-Clerc \cite{BC}, which says that the sharp results for
Ces\`aro summation on $L^p$ can be deduced from these asymptotic
estimates of  orthogonal projections. ( See also Sogge \cite{So2}
for the case of general compact manifolds.) For our study, while
we can obtain  global sharp   asymptotic  bounds for the $L^p(
h_\k^2; S^d)\rightarrow L^2(h_\k^2;S^d)$ norms of  the orthogonal
projection operators  of  $h$-harmonic expansions (see Theorem
\ref{thm:proj} in Section 3),  which  are in full analogy with
those of Sogge \cite{So1} for ordinary spherical harmonics,
seemingly, these global estimates  are not enough for the proof
of the uniform boundedness of  Ces\`aro means on weighted $L^p$.
In order to obtain our  main result on the boundedness of Ces\`aro
means, we have to replace the norm of the orthogonal projection
operator by a local estimate of the projection operator over a
spherical cap. (See Theorem \ref{thm:proj-cap} in Section 3.) The
latter  local result is substantially more difficult to establish,
since only a part of the proof can follow Sogge's strategy based
on Stein's theorem on analytic interpolation and the rest has to
rely on sharp pointwise local estimate of the kernels.

Analogues of our main results also hold for orthogonal expansions
on the unit ball and on the simplex for weight functions related
to $h_\k^2$, including in particular the Lebesgue measure (see
Section 2). In fact, they follow more or less  from the results
for $h$-harmonics.  In particular, the same condition
\eqref{critical-index} guarantees the convergence of the Ces\`aro
means in the corresponding weighted $L^p$ space.

The paper is organized as follows: The next section contains
preliminary, the main results are stated and discussed in Section
3. The local estimate of the projection operator is studied in
Section 4. The proof of the main result for the projection
operators on the sphere is given in Section 5, while the proof of
the main result for the Ces\'aro means on the sphere is presented
in Section 6. Finally, the results on the ball and on the simplex
are proved in Section 7.


\section{Preliminary}
\setcounter{equation}{0}

\subsection{$h$-spherical harmonics}

We restrict our discussion to $h_\k$ in \eqref{weight}. Unless otherwise
stated, the main reference for the material in this section is
\cite{DX}. An $h$-harmonic is a homogeneous polynomial $P$ that
satisfies the equation $\Delta_h P =0$, where $\Delta_h := \CD_1^2
+ \ldots + \CD_{d+1}^2$ and
$$
    \CD_i f(x): = \partial_i f(x) + \k_i \frac{f(x) - f(x - 2 x_i e_i)}
    {x_i},\   \  1\leq i\leq d+1,
$$
$e_1,\cdots, e_{d+1}$ denote the usual coordinate vectors in
$\mathbb{R}^{d+1}$.  The differential-difference operators
$\CD_1,\ldots,\CD_{d+1}$ are the Dunkl operators, which commute.
An $h$-harmonic is an orthogonal polynomial with respect to the
weight function $h_\k^2(x)$ on $S^d$. Its restriction on the
sphere is called a spherical $h$-harmonic. Let $\CH_n^{d}(h_\k^2)$
denote the space of spherical $h$-harmonics of degree $n$ on
$S^d$. It is known that $\dim
\CH_n^d(h_\k^2)=\binom{n+d+1}{n}-\binom{n+d-1}{n-2}$. The Hilbert
space theory shows that
$$
L^2(h_k^2; S^{d}) = \sum_{n=0}^\infty \CH_n^{d}(h_\k^2):  \qquad
       f = \sum_{n=0}^\infty \proj_n(h_\k^2; f),
$$
where $\proj_n(h_\k^2) : L^2(h_\k^2;S^{d}) \mapsto \CH_n^{d}(h_\k^2)$ is
the projection operator, which can be written as an integral operator
$$
   \proj_n(h_\k^2; f, x) = a_\k \int_{S^{d}} f(y) P_n(h_\k^2 ;x,y) h_\k^2(y)
   d\o(y),\   \    x\in S^d,
$$
where $d\o(y)$ denotes the usual Lebesgue measure on $S^d$, $a_\k$ is
the normalization constant, $a_\k^{-1} = \int_{S^d} h_\k^2(y) d\omega(y)$
and $P_n(h_\k^2)$ is the reproducing kernel of $\CH_n^{d}(h_\k^2)$. The
kernel satisfies an explicit formula
\begin{equation} \label{proj-kernel}
 P_n(h_\k^2; x,y) =  \frac{n+\l_\k}{\l_\k} V_\k \left[
      C_n^{\l_\k}( \la \cdot,y\ra )\right](x), \qquad \l_\k = \frac{d-1}{2}+|\k|,
\end{equation}
where $C_n^\l$ is the Gegenbauer polynomial of degree $n$ and $V_\k$
is the so-called intertwining operator defined by
\begin{equation} \label{eq:Vk}
V_\k f(x) = c_\k \int_{[-1,1]^{d+1} } f(x_1t_1,\ldots,x_{d+1}t_{d+1})
            \prod_{i=1}^{d+1} (1+t_i)(1-t_i^2)^{\k_i-1}dt,
\end{equation}
in which $c_\k$ is a constant such that $V_\k 1 =1$. If some
$\kappa_i =0$, then the formula  holds under the limit relation $$
 \lim_{\lambda \to 0} c_\lambda \int_{-1}^1 g(t)(1+t) (1-t)^{\lambda -1} dt
  = g(1).$$
 Clearly \eqref{proj-kernel}
is the same as \eqref{proj-Kernel}. The operator $V_\k$ is called an
intertwining operator since it satisfies $\CD_j V_\k = V_\k \partial_j$,
$1 \le j \le d+1$.

Let $w_\l(t): = (1-t^2)^{\l-1/2}$ on $[-1,1]$. The Gegenbauer
polynomials are orthogonal with respect to $w_\l$. The
intertwining operator can be used to define a convolution $f
\ast_\k g$ for $f \in L^1(h_\k^2;S^d)$  and $g \in L^1(w_{\l_k};
[-1,1])$ (\cite{X05a})
\begin{equation} \label{convolution}
f \ast_\k g(x) :=a_\k \int_{S^d} f(y) V_\k[g(\la x,\cdot \ra)](y)
          h_\k^2(y) d\omega(y).
\end{equation}
In particular, the projection operator $\proj_n(h_\k^2; f)$ can be
written as
\begin{equation} \label{projection}
  \proj_n(h_\k^2; f) =  f \ast_\k Z_n^\k, \quad\hbox{where}\quad Z_n^\k(t) :=
          \frac{n+\l_\k}{\l_k} C_n^{\l_k}(t).
\end{equation}
This convolution satisfies the usual Young's inequality (see
\cite[p.6, Proposition  2.2]{X05a}). 
For $\k =0$, $V_\k = id$, it becomes the classical
convolution on the sphere (\cite{CZ}).  For $f \in L^1(h_\k^2;S^d)$, we also
have (\cite{X05a})
\begin{equation} \label{projf*g}
\proj_n(h_\k^2;  f\ast_\k g) = b_{\l_\k} \int_0^\pi
     \frac{C_n^{\l_\k}(\cos \t)}{C_n^{\l_\k}(1)}
        g(\cos \t) (\sin \t )^{2\l_\k} d\t \, \proj_n(h_\k^2; f),
\end{equation}
where $b_{\l_\k}$ is the normalization constant of $w_{\l_\k} (t)$
on $[-1,1]$.

For $\d > -1$, the Ces\`aro $(C,\d)$ means of the $h$-harmonic expansion
is defined by
$$
 S_n^\d (h_\k^2;f,x) : = (A_n^\delta)^{-1}\sum_{j=0}^n
       A_{n-j}^\delta \proj_j(h_\k^2; f,x),
 \qquad A_{n-j}^\delta = \binom{n-j+\delta}{n-j}.
$$
The operator $S_n^\d(h_\k^2)$ can be written as a convolution,
\begin{equation*}
   S_n^\d (h_\k^2; f) = f \ast_\k K_n^\d(w_{\l_\k}), \qquad
              K_n^\d(w_{\l_\k}; t):= (A_n^\delta)^{-1}\sum_{j=0}^n
                 A_{n-j}^\delta  Z_j^{\k}(t).
\end{equation*}
Let  $K_n^\d(h_\k^2;x,y)$ denote the kernel of $S_n^\d (h_\k^2)$; then
$$
   K_n^\d(h_\k^2;x,y) = V_\k\left[K_n(w_{\l_\k}; \la x,\cdot\ra) \right](y).
$$

\subsection{Orthogonal expansions on the unit ball}
We denote the usual  Euclidean norm of $x=(x_1,\cdots,
x_{d})\in\mathbb{R}^{d}$ by
$\|x\|:=(x_1^2+\cdots+x_{d}^2)^{\f12}$. The weight functions we
consider on the unit ball $B^d = \{x: \|x\| \le 1\} \subset \RR^d$
are defined by
\begin{equation} \label{weightB}
 W_\k^B(x): = \prod_{i=1}^{d} |x_i|^{\k_i}(1-\|x\|^2)^{\k_{d+1}-1/2},
      \qquad \k_i \ge 0,  \quad x \in B^d,
\end{equation}
which is related to the $h_\k$ in \eqref{weight} by $h_\k^2(x, \sqrt{1-\|x\|^2})=
W_\k^B(x) /\sqrt{1-\|x\|^2}$, in which $1/\sqrt{1-\|x\|^2}$  comes from the
Jacobian of changing variables
\begin{equation}\label{B-S}
\phi: x\in B^d \mapsto (x,\sqrt{1-\|x\|^2}) \in S^{d}_+:=\{y \in S^d:y_{d+1} \ge 0\}.
\end{equation}
Furthermore, under the above changing variables, we have
\begin{equation}\label{BSintegral}
\int_{S^d} g(y) d\omega(y) = \int_{B^d} \left[
g(x,\sqrt{1-\|x\|^2}\,)+
     g(x,-\sqrt{1-\|x\|^2}\,) \right]\frac{dx}{\sqrt{1-\|x\|^2}}.
\end{equation}
The orthogonal structure is preserved under the mapping
\eqref{B-S} and the study of orthogonal expansions for $W_{\k}^B$
can be essentially reduced to that of $h_\k^2$. In fact, let
$\CV_n^d(W_{\k}^B)$ denote the space of orthogonal polynomials of
degree $n$ with respect to $W_{\k}^B$ on $B^d$. The orthogonal
projection, $\proj_n(W_{\k}^B; f)$,  of $f \in L^2(W_{\k}^B;B^d)$
onto $\CV_n^d(W_{\k}^B)$ can be expressed in terms of the
orthogonal projection of $F(x,x_{d+1}):= f(x)$ onto
$\CH_n^{d+1}(h_\k^2)$:
\begin{equation}\label{projBS}
   \proj_n(W_{\k}^B; f, x) = \proj_n(h_\k^2;  F, X), \qquad \text{with}\  \  X := (x,\sqrt{1-\|x\|^2}).
\end{equation}
This relation allows us to deduce results on the convergence of orthogonal
expansions with respect to $W_\k^B$ from that of $h$-harmonic expansions.

For $d =1$ the weight $W_\k^B$ in \eqref{weightB} becomes the weight function
\begin{equation}\label{GGweight}
  w_{\k_2,\k_1} (t) = |t|^{2 \k_1} (1- t^2)^{\k_2-1/2}, \qquad \k_i \ge 0, \quad
     t \in [-1,1],
\end{equation}
whose corresponding orthogonal polynomials, $C_n^{(\k_1,\k_2)}$, are
called generalized Gegenbauer polynomials, and they can be expressed
in terms of Jacobi polynomials,
\begin{align}\label{G-Gegen}
\begin{split}
C_{2n}^{(\lambda ,\mu )}(t) &=\frac{\left( \lambda +\mu \right)_{n}}
{\left( \mu +\frac{1}{2}\right)_{n}} P_{n}^{(\lambda -1/2,\mu-1/2)}
(2t^{2}-1), \\
C_{2n+1}^{(\lambda ,\mu )}(t) &=\frac{\left( \lambda +\mu
\right)_{n+1}} {\left( \mu
+\frac{1}{2}\right)_{n+1}}tP_{n}^{(\lambda -1/2,\mu+1/2)}
(2t^{2}-1),
\end{split}
\end{align}
where $(a)_n =a(a+1)\cdots (a+n-1)$.

\subsection{Orthogonal expansions on the simplex} The weight
functions we consider on the simplex $T^d =\{x:x_1\ge 0, \ldots, x_d\ge 0,
 1-|x| \ge 0\}$ are defined by
\begin{equation}  \label{weightT}
 W_\k^T(x) := \prod_{i=1}^{d} x_i ^{\k_i-1/2}(1-|x|)^{\k_{d+1}-1/2},
      \qquad \k_i \ge 0,
\end{equation}
where $|x| := x_1+ \cdots + x_d$. They are related to $W_\k^B$, hence to
$h_\k^2$. In fact, $W_\k^T$ is exactly the product of the weight function
$W_\k^B$ under the mapping
\begin{equation}\label{psi}
\psi: (x_1,\ldots,x_d) \in T^d \mapsto (x_1^2, \ldots, x_d^2) \in B^d
\end{equation}
and the Jacobian of this change of variables. Furthermore,
the change of variables shows
\begin{equation} \label{T-B}
\int_{B^d} g(x_1^2,\ldots,x_d^2) dx = \int_{T^d} g(x_1,\ldots,x_d)
            \frac{dx}{\sqrt{x_1\cdots x_d}}.
\end{equation}
The orthogonal structure is preserved under the mapping
\eqref{psi}. Let $\CV_n^d(W_\k^T)$ denote the space of orthogonal
polynomials of degree $n$ with respect to $W_\k^T$ on $T^d$. Then
$R \in \CV_n^d(W_\k^T)$ if and only if $R\circ \psi \in
\CV_{2n}^d(W_\k^B)$. The orthogonal projection, $\proj_n(W_\k^T;
f)$,  of $f \in L^2(W_\k^T;T^d)$ onto $\CV_n^d(W_\k^T)$ can be
expressed in terms of the orthogonal projection of $f \circ \psi$
onto $\CV_{2n}^d(W_\k^B)$:
\begin{equation}\label{projTB}
  \left ( \proj_n(W_\k^T; f) \circ \psi \right)(x) =
     \frac{1}{2^d} \sum_{\ve \in \ZZ_2^d} \proj_{2n}(W_\k^B; f\circ \psi,
     x\ve),
\end{equation}
 The fact that $\proj_n(W_\k^T)$ of degree $n$ is related to
$\proj_{2n}(W_\k^B)$ of  degree $2n$ suggests that some properties
of the orthogonal expansions on $B^d$ cannot be transformed
directly to those on $T^d$. We will also need the explicit formula
for the kernel, $P_n(W_\k^T; x,y)$, of $\proj_n(W_\k^T; f)$, which
can be derived from \eqref{proj-kernel} and the quadratic
transform between Gegenbauer and Jacobi polynomials,
\begin{align} \label{proj-kernelT}
  & P_n(W_\k^T; x,y) =  \frac{(2n+\l_\k)\Gamma(\frac12)\Gamma(n+\l_k)}
      {\Gamma(\l_\k+1)\Gamma(n+\frac12)} \\
    & \qquad \times c_\k \int_{[-1,1]^{d+1}}
        P_n^{(\l_k-\frac12, - \f 12)} \left( 2z(x,y,t)^2-1\right)
                \prod_{i=1}^{d+1} (1-t_i^2)^{\k_i-1} d t, \notag
\end{align}
where $z(x,y,t) =\sqrt{x_1 y_1}\, t_1 + \ldots + \sqrt{x_d y_d}\,
 t_d+ \sqrt{1-|x|} \sqrt{1-|y|}\, t_{d+1}.$

We will also denote the Ces\`aro means for orthogonal expansions with
respect to a weight function $W$ as $S_n^\d(W;f)$ and denote their
kernel as $K_n^\d(W; x,y)$, where $W$ is either $W_\k^B$ or $W_\k^T$.

\subsection{Some estimates}

Throughout this paper we denote by $c$ a generic constant that may
depend on fixed parameters such as $\k$, $d$ and $p$, whose value may
change from line to line. Furthermore we write $A \sim B$ if $A \ge c B$
and $B \ge c A$.

Let $d (x,y) := \arccos \la x ,y \ra$ denote the geodesic
distance of $x,y \in S^d$. For $0 \le \theta  \le \pi$, the set
$$
   c(x,\theta) := \{y \in S^d:d(x,y)\le \theta\} =
         \{y \in S^d: \langle x,y\rangle \ge \cos \theta\}
$$
is called the spherical cap centered at $x$ with radius $\t$. It is shown
in \cite{D} that $h_\k$ is a doubling weight and, furthermore, the
following estimate holds:

\begin{lem}\label{lem:doubling}
For $0 \le \t \le \pi$ and $x=(x_1,\cdots, x_{d+1})\in S^d$,
\begin{equation}\label{eq:doubling}
   \int_{c(x,\t)}  h_\k^2(y) d\o(y) \sim 
     \t^d \prod_{j=1}^{d+1} (|x_j|  +\t)^{2\k_j},
\end{equation}
where the constant of equivalence depends only on $d$ and $\k$.
\end{lem}

We refer to the remarkable paper \cite{MT} of  Mastroianni  and
Totik for various polynomial inequalities with doubling weights.

The Jacobi polynomials $P_n^{(\a,\b)}$ are orthogonal with respect to
the weight
$$
   w^{(\a,\b)}(t) := (1-t)^\a (1+t)^{\be}, \qquad  \quad t \in [-1,1].
$$
We will need the following estimate from \cite[p. 169]{Szego}:

\begin{lem} \label{lem:3.2}
For  $\alpha\ge \be$ and $t \in [0,1]$,
\begin{equation} \label{Est-Jacobi}
|P_n^{(\alpha,\beta)} (t)| \le c n^{-1/2} (1-t+n^{-2})^{-(\alpha+1/2)/2}.
\end{equation}
The estimate on $[-1,0]$ follows from the fact that $P_n^{(\alpha,\beta)}
(t) = (-1)^nP_n^{(\beta, \alpha)} (-t)$.
\end{lem}

We will also need the estimate of the $L^p$ norm for the Jacobi polynomials
(\cite[p. 391]{Szego}: for $\a,\b, \mu > -1$ and $p >0$,
\begin{align}\label{JacobiLp}
\int_0^1 \left| P_n^{(\a,\b)}(t)\right|^p(1-t)^\mu dt \sim
    \begin{cases} n^{\a p -2\mu-2}, & p  >  p_{\a,\mu}, \\
         n^{-\f p2}\log n, & p  = p_{\a,\mu}, \\
         n^{-\f p2}, & p  <  p_{\a,\mu}.
\end{cases}
\quad p_{\a,\mu}:= \frac{2\mu+2}{\a +\f12}.
\end{align}

Recall $|\k|=\k_1+\cdots+\k_{d+1}$, $V_\k$ defined by
\eqref{eq:Vk} and the formula \eqref{projection}. The following
lemma was proved in \cite[Theorem 3.1]{DaiX}.

\begin{lem} \label{lem:VJacobi}
Assume  $\al\ge \max\{\be, |\k|-\f12\}$. Then for $x,y \in S^d$
\begin{align} \label{VJacobi}
  & \left| \int_{[-1,1]^{d+1}} P_n^{(\a,\b)}(x_1y_1t_1+ \ldots +x_{d+1}y_{d+1} t_{d+1})
       \prod_{j=1}^{d+1} (1+t_j)(1-t_j^2)^{\k_j-1} dt \right| \\
 & \qquad\quad \leq c
       n^{\al-2|\k|}\f{ \prod_{j=1}^{d+1}(|x_jy_j|+n^{-1}
       \|\bar{x}-\bar{y}\|+n^{-2})^{-\k_j}}
        {(1+n d(\bar{x},\bar{y}))^{\al+\f12-|\k|}}, \notag
\end{align}
where and throughout,  $\bar z = (|z_1|, \ldots,  |z_{d+1}|)$ for
$z=(z_1,\cdots, z_{d+1})\in S^d$.
\end{lem}

This lemma plays an essential role in the proof of a sharp pointwise
estimate for the kernel $K_n^\d(h_\k^2;f)$ in \cite{DaiX}. For the present
paper we will only need the pointwise estimate for $P_n(h_\k^2;x,y)$:

\begin{lem} \label{thm:S-estimate}
Let  $x, y \in S^{d}$. Then
\begin{align} \label{eq:S-est}
|P_n(h_\kappa^2; x,y)| \le \;  c
\frac{ \prod_{j=1}^{d+1}(|x_jy_j|+n^{-1}\|\bar x- \bar y\|+n^{-2})^{-\kappa_j}}
 {n^{-(d-1)/2} (\|\bar x - \bar y\|+ n^{-1})^{(d-1)/2} }.
\end{align}
\end{lem}

The kernel $P_n(W_\k^B)$ can be derived from \eqref{eq:S-est} and
will not be needed. We will need, however, the estimate for the
kernel $P_n(W_\k^T)$, which is also proved in \cite{DaiX}.

\begin{lem} \label{thm:T-estimate}
For $x =(x_1,\ldots,x_d) \in T^d$ and $y =(y_1,\ldots,y_d) \in T^d$,
\begin{align} \label{eq:T-est}
 |P_n (W_{\kappa}^T; x,y)| \le & \; c
  \frac{ \prod_{j=1}^{d+1} (\sqrt{x_j y_j} + n^{-1} \|\xi- \zeta\|+
    n^{-2})^{- \kappa_j} }
 { n^{ - (d-1)/2} (\|\xi - \zeta\|+ n^{-1})^{(d-1)/2}},
\end{align}
where $
\xi := (\sqrt{x_1}, \ldots, \sqrt{x_d}, \sqrt{x_{d+1}})$, 
$\zeta := (\sqrt{y_1}, \ldots, \sqrt{y_d}, \sqrt{y_{d+1}})
$
with $x_{d+1}:=1-|x|$ and $y_{d+1}:=1-|y|$.
\end{lem}


\section{Main Results}
\setcounter{equation}{0}

\subsection{$h$-harmonic expansions}

For $h_\k$ defined in \eqref{weight}, we denote the $L^p$ norm of
$L^p(h_\k^2;S^d)$ by $\|\cdot \|_{\k,p}$,
$$
   \|f\|_{\k,p} : = \left( a_\k \int_{S^d} |f(y)|^p h_\k^2(y) d\o(y) \right)^{1/p}
$$
for $1 \le p  < \infty$ and with the usual understanding that it is the
uniform norm on $S^d$ when $p=\infty$.  Recall that
$$
   \s_\k : = \tf{d-1}2 + |\k| - \k_{\rm min} \quad \hbox{with} \quad
      \k_{\rm min}:=\min_{1\le j \le d+1} \k_j.
$$

Our main results on the Ces\`aro summation of $h$-harmonic
expansions are the following two theorems:

\begin{thm}\label{thm:Cesaro}
Suppose that $f\in L^p(h_\k^2; S^d)$, $1\leq p\leq  \infty$,
$|\f1p-\f12|\ge \f 1{2\sa_\k+2}$ and
\begin{equation}\label{delta(p)}
\da>\da_\k(p):=\max\{(2\sa_\k+1)|\tf1p-\tf12|-\tf12,0\}.
\end{equation}
Then $S_n^\d (h_\k^2;f)$ converges to $f$ in $L^p(h_\k^2;S^d)$ and
$$
  \sup_{n\in\mathbb{N}}\|S_n^\da(h_\k^2;f)\|_{\k,p}\leq c \|f\|_{\k,p}.
$$
\end{thm}

\begin{thm} \label{thm:Cesaro2}
Assume  $1\leq p\leq \infty$ and  $0<\d\leq \da_\k (p) $. Then
there exists a function $f \in L^p(h_\k^2;S^d)$ such that $S_n^\d
(h_\k^2; f )$ diverges  in $L^p(h_\k^2;S^d)$.
\end{thm}

For $\k =0$, $h_\k(x) \equiv 1$ and the spherical $h$-harmonic becomes
the ordinary spherical harmonics. Hence Theorem \ref{thm:Cesaro} is the
complete analogue of the  Sogge theorem, while Theorem \ref{thm:Cesaro2}
is the analogue of \cite[Theorem 5.2]{BC} for spherical harmonics.

For the projection operator $\proj_n(h_\k^2;f)$ we have the following
theorem which is a complete analogue of a theorem due to Sogge
\cite{So1} for spherical harmonics.

\begin{thm} \label{thm:proj}
Let $d \ge 2$ and  $n \in \mathbb{N}$. Then
\begin{enumerate}
\item[(i)] for $1 \leq  p \le \frac{2(\sa_k+1)}{\sa_k+2}$,
$$
 \left \|\proj_n(h_\k^2; f)\right \|_{\k,2}
  \le c n^{\da_\k(p)}
       \|f\|_{\k,p},
$$ with $\da_k(p)$  given in (\ref{delta(p)});
\item[(ii)] for $\frac{2(\sa_k+1)}{\sa_k+2} \le p \le 2$,
$$
  \left  \|\proj_n(h_\k^2; f)\right \|_{\k,2} \le c n^{\sa_k (\frac1p-\frac12)} \|f\|_{\k,p}.
$$
\end{enumerate}
Furthermore, the estimate (i) is sharp.
\end{thm}

The estimate in (ii) is sharp if $\k =0$ as shown in \cite{So1}. We expect
that it is also sharp for $\k \ne 0$ but could not prove it at this moment.
For further discussion on this point, see Remark \ref{rem5.1} in Section 5.

For the spherical harmonics, the above theorem is enough for the
proof of the boundedness of the Ces\`aro means. (See \cite{BC} and
\cite{So2}.)   For $h$-harmonics, however, a stronger result is
needed since $\da_\k(p)>\da(p):=\max\{d|\f 1p-\f12|-\f12,0\}$.

\begin{thm}\label{thm:proj-cap}
Suppose  that $1\leq p \leq \f{2\sa_\k+2}{\sa_k+2}$ and $f$ is
supported in a spherical cap  $c(\varpi, \t)$ with $\t \in
(n^{-1},\pi]$ and  $\varpi\in S^d$. Then
$$
\left \|\proj_n (h_\k^2; f)\right\|_{\k, 2} \leq c n^{\da_\k(p)}
  \t^{\da_\k(p)+\f12} \left[  \int_{c(\varpi, \t)}
         h_\k^2(x)\, d\og(x)\right]^{\f 12 -\f1p}\|f\|_{\k,p}.
$$
 \end{thm}

The above theorems on the projection operators will be proved in
Section 5 and the theorems on Ces\'aro means will be  proved in
Section 6.


\subsection{Orthogonal expansions on the ball and on the simplex}

Let $\Omega^d$ stand for either $B^d$ or $T^d$ and $W_\k^\Omega$
stand for either $W_\k^B$ or $W_\k^T$, respectively. We denote the
$L^p$ norm of $L^p(W_\k^\Omega;\Omega^d)$ by
$\|\cdot \|_{W_\k^\Omega,p}$,
$$
   \|f\|_{W_\k^\Omega,p} : = \left( a_\k^\Omega
           \int_{\Omega^d} |f(y)|^p W_\k^\Omega(y) dy \right)^{1/p}
$$
for $1 \le p  < \infty$ and with the usual understanding that it becomes
the uniform norm on $\Omega^d$ when $p=\infty$.

Our main results on the Ces\`aro summation of orthogonal
expansions on $B^d$ and $T^d$ are the following two theorems:

\begin{thm}\label{thm:CesaroBT}
Suppose that  $f\in L^p(W_\k^\Omega;\Omega^d)$, $1\leq p\leq
\infty$,  $|\f1p-\f12|\ge \f 1{2\sa_\k+2}$ and
$$
\da>\da_\k(p):=\max\{(2\sa_\k+1)|\tf1p-\tf12|-\tf12,0\}.
$$
Then $S_n^\d (W_\k^\Omega;f)$ converges to $f$ in
$L^p(W_\k^\Omega;\Omega^d)$ and
$$
 \sup_{n\in\mathbb{N}}\|S_n^\da(W_\k^\Omega;f)\|_{W_\k^\Omega,p}
        \leq c \|f\|_{W_\k^\Omega,p}.
$$
\end{thm}

\begin{thm} \label{thm:CesaroBT2}
Assume $1\leq p\leq \infty$ and  $0<\d\leq \da_\k (p) $. Then
there exists a function $f \in L^p(W_\k^\Omega; \Omega^d)$ such
that $S_n^\d (W_\k^\Omega;f)$ diverges  in
$L^p(W_\k^\Omega;\Omega^d)$.
\end{thm}

For $d =1$ and $\Omega = T^1=[0,1]$, these theorems become results for
the Jacobi polynomial expansions (\cite{Ch-Mu}). For $d=1$ and
$\Omega = B^1=[-1,1]$, these theorems become results for the generalized
Gegenbauer polynomial expansions with respect to $w_{\k_1,\k_2}$
in \eqref{GGweight}, which appear to be new if $\k_1 \ne 0$ while
the case $\k =0$ corresponds to Gegenbauer polynomial expansions
\cite{AH,BC}.  We state the result as follows:

\begin{cor}\label{thm:CesaroGG}
Suppose that $1\leq p\leq \infty$,  $|\f1p-\f12|\ge \f
1{2\max\{\l,\mu\}+2}$ and
$$
\da>\da(p):=\max\{(2\max\{\l,\mu\}+1)|\tf1p-\tf12|-\tf12,0\}.
$$
Then $S_n^\d (w_{\l,\mu};f)$ converges to $f$ in
$L^p(w_{\l,\mu};[-1,1])$ and
$$
 \sup_{n\in\mathbb{N}}\|S_n^\da(w_{\l,\mu};f)\|_{w_{\l,\mu},p}
        \leq c \|f\|_{w_{\l,\mu},p}.
$$
Furthermore, the condition $\da > \da(p)$ is sharp.
\end{cor}

The result analogous to Theorem \ref{thm:proj} also holds for the
projection operator.

\begin{thm} \label{thm:projBT}
Let $d \ge 2$ and $n \in \mathbb{N}$. Then
\begin{enumerate}
\item[(i)] for $1 \leq  p \le \frac{2(\sa_k+1)}{\sa_k+2}$,
$$
 \left \|\proj_n(W_\k^\Omega; f)\right \|_{W_\k^\Omega,2} \le
    c n^{\da_\k(p)}
    \|f\|_{W_\k^\Omega,p};
$$
\item[(ii)] for $\frac{2(\sa_k+1)}{\sa_k+2} \le p \le 2$,
$$
  \left  \|\proj_n(W_\k^\Omega; f) \right \|_{W_\k^\Omega,2} \le
         c n^{\sa_k (\frac1p-\frac12)} \|f\|_{W_\k^\Omega,p}.
$$
\end{enumerate}
Furthermore, the estimate in (i) is sharp.
\end{thm}

An analogue of Theorem \ref{thm:proj-cap} also holds. We state only the
one for $W_\k^T$ for which we define an distance on $T^d$,
$$
     d_T(x,y): = \arccos ( \sqrt{x_1 y_1} +  \ldots + \sqrt{x_d y_d} +
          \sqrt{1-|x|}\sqrt{1-|y|} ),
$$
where $|z|=|z_1|+\cdots +|z_{d}|$ for $z\in\mathbb{R}^d$.  The
analogue of the spherical cap on $T^d$ is defined as
$c_T(x,\theta) := \{y \in T^d:d_T(x,y)\le \theta\}$.

\begin{thm}\label{thm:proj-capT}
Suppose $1\leq p \leq \f{2\sa_\k+2}{\sa_k+2}$ and  $f$ is
supported in the set $c_T(x, \t)$ with $\t \in (n^{-1},\pi]$ and
$x\in T^d$. Then
$$
  \left \|\proj_n (W_\k^T; f) \right\|_{W_\k^T, 2} \leq c n^{\da_\k(p)}
       \t^{\da_\k(p)+\f12} \[  \int_{{c_{} }_T(x, \t)}
         W_\k^T(y)\, dy \]^{\f 12 -\f1p}\|f\|_{W_\k^T,p}.
$$
\end{thm}

The analogue result for $B^d$ holds with $c_T(x,\theta)$ replaced by
$c_B(x,\theta)$ defined in terms of $d_B (x,y) : = \arccos (\la x, y\ra +
  \sqrt{1-\|x\|^2}\sqrt{1-\|y\|^2} )$.

 These results will be proved in Section 7.


\section{Local estimate of projection operator}
\setcounter{equation}{0}

The main effort in the proof of Thereom \ref{thm:proj-cap}, giving in
the next section, lies in proving the following local estimate of the
projection operator.

\begin{thm}\label{thm:proj-local} Let $\nu: =\f{2\sa_\k +2}{\sa_\k+2}$ and
$\nu':=\f{\nu}{\nu-1}$. Let $f$  be a function supported in a
spherical cap $c(\varpi, \t)$ with $\t \in(n^{-1}, 1/(8d)]$ and
$\varpi\in S^d$. Then
\begin{align*}
 \left \|\proj_n(h_\k^2;f)\chi_{c(\varpi,\t)} \right \|_{\k,\nu'}
     \leq c n^{\f{\sa_\k}{1+\sa_\k}}  \t^{\f{2\sa_\k+1}{\sa_\k+1}}
       \[  \int_{c(\varpi,\t)} h_\k^2(x) d\o(x)\]^{1 -\f2\nu} \|f\|_{\k,\nu}.
\end{align*}
\end{thm}

Here $\chi_E$ denotes the characteristic function of the set $E$.
Note the norm of the left hand side is taken over $c(\varpi,\t)$,
so that the above estimate is a local one.

Throughout this section, we shall fix the spherical cap $c(\varpi,\t)$.
Without loss of generality, we may assume $\varpi=(\varpi_1,\cdots,\varpi_{d+1})$
satisfying $|\varpi_k| \ge 4\t $ for $1\leq k\leq v$ and $|\varpi_k|<4\t$
for $v<k\leq d+1$. Accordingly, we define
\begin{equation} \label{def:gamma}
   \g = \g_\varpi:=\begin{cases}
   0,&\   \    \  \text{if $v=d+1$};\\
   \displaystyle\sum_{i=v+1}^{d+1}\k_i,&\    \   \text{if $v<d+1$}.\end{cases}
\end{equation}
Since  $\t \in (0,1/(8d)]$ and $\varpi\in S^d$, it follows that
\begin{equation} \label{2-5}
  0 \le \ga \leq |\k|-\min_{1\leq i\leq d+1} \k_i=\sa_\k-\tf{d-1}2.
\end{equation}

The proof of Theorem \ref{thm:proj-local} consists of two cases,
one for $\g < \s_\k -\frac{d-1}{2}$ and the other for $\g = \s_\k
-\frac{d-1}{2}$, using different methods.


\subsection{Proof of Theorem \ref{thm:proj-local}, case I: $\g < \s_\k -\frac{d-1}{2}$}
The proof is long and will be divided into several subsections.

\subsubsection {Decomposition of the projection operator} Recall
$\l_\k=\f {d-1}2+|\k|$.  Let $\xi_0\in C^\infty [0,\infty)$ be
such that $\chi_{[0,1/2]}(t)\leq \xi_0(t) \leq \chi_{[0,1]}(t)$,
and define $\xi_1(t):=\xi_0(t/4)-\xi_0(t)$. Evidently
$\text{supp}\, \xi_1\subset (1/2, 4)$ and $\xi_0(t)+
\sum_{j=1}^\infty \xi_1(4^{-j+1}t)=1$ whenever $ t\in [0,\infty)$.
Define, for $u\in [-1,1]$,
\begin{align*}
& C_{n,0}(u): = \f{n+\l_\k}{\l_\k}C_n^{\l_\k}(u) \xi_0\(n^2(1-u^2)\) \\
& C_{n,j}(u):= \f{n+\l_\k}{\l_\k}C_n^{\l_\k}(u)
\xi_1\(\f{n^2(1-u^2)}{4^{j-1}}\), \quad
   j = 1,2,\ldots, L_n,
\end{align*}
where $L_n:=\lfloor \log_2 n \rfloor + 2$. By \eqref{projection},
$\proj_n(h_\k^2;f)$ can be decomposed as
\begin{equation}\label{2-2}
\proj_n(h_\k^2; f)= \sum_{j=0}^{L_n} Y_{n,j}f, \quad\hbox{where}\quad
     Y_{n,j} f : = f \ast_\k C_{n,j}.
\end{equation}
By the definition of the convolution, the kernel of $Y_{n,j}$ is
 $V_\k [C_{n,j}(\la x, \cdot \ra)](y)$.

\subsubsection{Estimates of the kernels  $V_\k\[ C_{n,j}\la x, \cdot\ra \](y)$
and $L^\infty$ estimate}

\begin{defn}\label{def-2-3}
Given $n, v\in\mathbb{N}_0$, and  $\mu \in \mathbb{R}$,  we say a
continuous function  $F: [-1,1]\to\mathbb{R}$ belongs to the class
$\CS_n^v (\mu)$ if there exist functions $F_j$, $j=0,1,\cdots, v$
on $[-1,1]$ such that $F_j^{(j)}(t)=F(t)$, $t\in [-1,1]$, $0\leq
j\leq v$, and
\begin{equation}\label{2-3}
|F_j(t)| \le  n^{-2j+\mu} \left(1+n
\sqrt{1-|t|}\right)^{-\mu-\f12+j},\ \ \ t\in [-1,1],\   \ \
j=0,1,\cdots, v.
\end{equation}
\end{defn}

By \eqref{Est-Jacobi} and the following well known formula \cite[(4.21.7)]{Szego}
\begin{equation} \label{d-Jacobi}
    \frac{d}{dt} P_n^{(\a,\b)}(t ) = \tfrac12 (n+\a+\b+1) P_{n-1}^{(\a+1,\b+1)}(t ),
\end{equation}
it follows that $c_{v,\al} P_n^{(\a,\b)} \in \CS_n^v(\a)$ for all
$v \in \NN_0$ whenever $\al\ge \be$.

\begin{lem}\label{lem-2-2}
Assume that $\d=(\d_1,\cdots,\d_m) \in \RR^m$
satisfying $\min_{1\leq j\leq m} \da_j> 0$ and $\mu \in \RR$.  Let $F\in
\CS_n^{v}(\mu)$ with $v$ being an integer satisfying $v \ge 2m +
\sum_{j=1}^m\da_j  +|\mu|$. Let $\xi$ be a $C^\infty$ function,
supported in $[-8,8]$ and equal to constant in a neighborhood of
$0$. For $\rho\in (n^{-1},4]$, define
$$
G(u) : = F(u) \xi\(\f{1-u^2}{\rho^2}\),\qquad  u\in [-1,1].
$$
Then for $s \in [-1,1]$ and $a=(a_1,\cdots, a_m)\in [-1,1]^m$ satisfying
$\sum_{j=1}^m |a_j|+|s|\leq 1$,
\begin{align}
   & \left |\int_{[-1,1]^m} G \Bl(\sum_{j=1}^m a_jt_j+s \Br)
       \prod_{j=1}^m (1-t_j^2)^{\da_j-1}(1+t_j)\, dt_j
       \right | \label{2-4}\\
   & \qquad\quad \leq c  n^{-\f12-|\da|}\rho^{|\da|-\mu-\f12} \prod_{j=1}^m
        (|a_j|+n^{-1}\rho)^{-\da_j},\notag
 \end{align}
 where $|\da|=\sum_{j=1}^m\da_j$.
\end{lem}

\begin{proof}
Without loss of generality, we may assume that  $|a_j|\ge
n^{-1}\rho$ for $1\leq j\leq m$, since otherwise we can modify the
proof by replacing $s$ with
$$s+\sum_{\{j: |a_j|< n^{-1}\rho\}} a_j
t_j.$$

Let  $\eta_0\in C^\infty(\mathbb{R})$ be such that $\eta_0(t)=1$
for $|t|\leq \f12$ and $\eta_0(t)=0$ for $|t|\ge 1$, and  let
$\eta_1(t)=1-\eta_0(t)$. Set $
 B_j : = \f    \rho{n|a_j|}$,    $ j=1,\cdots,m$.
Given $\va:=(\va_1,\cdots,\va_m)\in\{0,1\}^m$, we define
$\psi_\va:\   \ [-1,1]^m\to \mathbb{R}$ by
 $$
 \psi_\va(t) :=\xi\(\f{1-(\sum_{j=1}^m a_j t_j+s)^2}{\rho^2}\)
  \prod_{j=1}^m \eta_{\va_j}\Bl(\f{1-t^2_j}{B_j}\Br)(1+t_j) (1-t_j^2)^{\da_j-1},
$$
where $t=(t_1,\cdots, t_m)$.  We then split the integral in
(\ref{2-4}) into a finite sum:
\begin{align*}
 \sum_{\va \in \{0,1\}^m} \int_{[-1,1]^m} F\Bl(\sum_{j=1}^m a_jt_j+s\Br)
    \psi_\va(t) \, dt  =: \sum_{\va\in \{0,1\}^m}  J_\va.
\end{align*}
Thus, it is sufficient to prove that  each term $J_\va$ in the above sum
satisfies the desired inequality. By symmetry and Fubini's theorem,  we
need only  to consider the case when $\va_1=\cdots=\va_{m_1}=0$ and
$\va_{m_1+1}=\cdots=\va_{m}=1$ for some $0\leq m_1\leq m$.

Let $m_1$ and $\ve$ be fixed as in the last line. Fix $(t_1,\cdots, t_{m_1})
\in [-1,1]^{m_1}$ momentarily, and write  $s_1=\sum_{j=1}^{m_1}a_j t_j+s$.
Define
$$
\phi(t):  =\xi\(\f{1-(\sum_{j=1}^m a_j
t_j+s)^2}{\rho^2}\)\prod_{j=m_1+1}^m
\eta_{1}\Bl(\f{1-t^2_j}{B_j}\Br)(1+t_j) (1-t_j^2)^{\da_j-1}.
$$
Since the support set of each $\eta_{1}\(\f{1-t^2_j}{B_j}\)$ is a
subset of $\{t_j:  |t_j|\leq 1-\f14 B_j\}$, we can use integration
by parts $|\mathbf{l}|=\sum_{j={m_1+1}}^m \ell_j$ times to obtain
\begin{align*}
&\Bl |\int_{[-1,1]^{m-m_1}} F\Bl(\sum_{j=m_1+1}^m a_jt_j+s_1\Br)\phi(t)\,
  dt \Bl |\\
&  = \prod_{j=m_1+1}^m |a_j|^{-\ell_j}\Bl |\int_{[-1,1]^{m-m_1}}
F_{|\mathbf{l}|}\Bl(\sum_{j=m_1+1}^m a_jt_j+s_1\Br)
 \f{\p^{|\mathbf{l}|}\phi(t)}{\p^{\ell_{m_1+1}} t_{m_1+1}\cdots\p
 ^{\ell_{m}} t_m}\, dt
 \Br |\\
&\leq \prod_{j=m_1+1}^m |a_j|^{-\ell_j}\int_{[-1,1]^{m-m_1}}
 \Bl|F_{|\mathbf{l}|} \Bl(\sum_{j=m_1+1}^m
 a_jt_j+s_1\Br)\Br|\Bl|
 \f{\p^{|\mathbf{l}|}\phi(t)}{\p^{\ell_{m_1+1}} t_{m_1+1}\cdots\p^{\ell_{m}} t_m}\Br|\, dt,
\end{align*}
where $F_{|\mathbf{l}|}^{(|\mathbf{l}|)}=F$ is as in Definition
\ref{def-2-3}, and  $\mathbf{l}=(\ell_{m_1+1},\cdots, \ell_m)\in
\mathbb{N}^{m-m_1}$ satisfies    $\ell_j>\da_j$ and
$|\mathbf{l}|\ge \mu+\f12$.  Since $\xi$ is supported in $(-8,8)$,
the integrand of the last integral is zero unless
\begin{align} \label{eq:4.4}
 8\rho^2 & \ge 1- \Bl|\sum_{k=m_1+1}^m a_k t_k +s_1\Br| \\
    & \ge 1-  \sum_{k=m_1+1}^m |a_k|-|s_1|+(1-|t_j|)|a_j|
    \ge |a_j|(1-|t_j|), \notag
\end{align}
for all $m_1+1\leq j\leq m$; that is, $\f {|a_j|}{\rho^2} \leq 8
(1-|t_j|)^{-1}$ for $j=m_1+1,\cdots, m$. Also, recall that $\xi$
is constant near $0$. Hence, taking the $k$-th partial derivative
with respect to $t_j$, the $\xi$ part of $\phi$ is bounded by $c
(1-t_j)^{-k}$. So is the same derivative of the $\eta_1$ part of
$\phi$ since $B_j^{-1} \le (1-t_j^2)^{-1}$ in the support of
$\eta_1'$. Consequently, by the Lebnitz rule, we conclude
$$
\Bl| \f{\p^{|\mathbf{l}|}\phi(t)}{\p^{\ell_{m_1+1}}
t_{m_1+1}\cdots\p^{\ell_m} t_m}\Br|\,
 \leq c \prod_{j=m_1+1}^m (1-|t_j|)^{\da_j-\ell_j-1}
$$
in the support of the integrand. Next, since $\rho\ge n^{-1}$ and
$|\mathbf{l}|\ge \mu+\f12$, \eqref{eq:4.4} together with
\eqref{2-3} implies
$$
\Bl|F_{|\mathbf{l}|}\Bl(\sum_{k=m_1+1}^m
 a_kt_k+s_1\Br)\Br|\leq c
 n^{-\f12-|\mathbf{l}|}\rho^{-\mu-\f12+|\mathbf{l}|}.
$$
  It follows that
\begin{align*}
&\int_{[-1,1]^{m-m_1}} \Bl|F_{|\mathbf{l}|}(\sum_{j=m_1+1}^m
 a_jt_j+s_1)\Br|\Bl|  \f{\p^{|\mathbf{l}|}\phi(t)}{\p^{\ell_{m_1+1}}
    t_{m_1+1}\cdots\p^{\ell_m} t_m}\Br|\, dt\\
 &\qquad\leq c    n^{-\f12-|\mathbf{l}|}\rho^{-\mu-\f12+|\mathbf{l}|}
     \prod_{j=m_1+1}^m\int_{0}^{1-\f{B_j}4}
    (1-t_j)^{\da_j-\ell_j-1}\, dt_j\\
 & \qquad \leq c n^{-\f12-|\mathbf{l}|}\rho^{-\mu-\f12+|\mathbf{l}|}
    \prod_{j=m_1+1}^m B_j^{\da_j-\ell_j}\\
& \qquad  \leq c  n^{-\f12-\al}\rho^{\al-\mu-\f12}
\prod_{j=m_1+1}^m
    |a_j|^{\ell_j-\da_j},
\end{align*}
where $\al=\sum_{j=m_1+1}^m \da_j$. Thus, since
$$
\psi_\va(t)  =\phi(t)\prod_{j=1}^{m_1}
\eta_{0}\Bl(\f{1-t^2_j}{B_j}\Br)(1+t_j) (1-t_j^2)^{\da_j-1},
$$
and $\eta_0\(\f{1-t_j^2}{B_j}\)$ is supported in $\{t_j:\  \ 1-B_j \leq
|t_j|\leq 1\}$, integrating  with respect to $t_1, \cdots, t_{m_1}$
over $[-1,1]^{m_1}$ yields
\begin{align*}
J_\va & \leq \int_{[-1,1]^{m_1}}\Bl|\int_{[-1,1]^{m-m_1}}
   F \Bl(\sum_{j=1}^m a_jt_j+s\Br)\phi(t)\, dt_{m_1+1}\cdots dt_m
   \Br| \\
&\qquad\quad \times  \prod_{j=1}^{m_1} \eta_{0}\(\f{1-t^2_j}{B_j}\)(1+t_j)
  (1-t_j^2)^{\da_j-1}dt_j\\
&\leq c
    n^{-\f12-\al}\rho^{\al-\mu-\f12} \prod_{j=m_1+1}^m
   |a_j|^{-\da_j}\prod_{j=1}^{m_1}\int_{
   1-B_j\leq |t_j|\leq 1}(1-|t_j|)^{\da_j-1}\, dt_j\\
&  \leq c n^{-\f12-|\da|}\rho^{|\da|-\mu-\f12} \prod_{j=1}^m
     |a_j|^{-\da_j},
\end{align*}
where we have used $|a_j|^{\ell_j} \le 1$ in the second step.
This completes the proof.
\end{proof}

Using the relation between the Gegenbauer  and the Jacobi polynomials,
we have
$$
C_{n,j}(u)=a_n P_n^{(\l_\k-\f12, \l_\k-\f12)}(u)\xi\(\f
{1-u^2}{(2^{j-1}/n)^2}\),
$$
where $\xi=\xi_1$ or $\xi_0$, and $|a_n|\leq c n^{\l_\k+\f12}$.
Hence, using the fact that $c_{v,\k} P_n^{(\l_\k-\f12,
\l_\k-\f12)}\in \CS_n^v(\l_\k-\f12)$ for all $v\in\mathbb{N}$,
Lemma \ref{lem-2-2} has the following corollary.

\begin{cor}\label{lem-2-3}
For $x, y\in S^d$ and $j=1,2,\ldots, L_n$,
$$
\left |V_\k\Bl[ C_{n,j}(\la x, \cdot\ra)\Br](y)\right|\leq c n^{d-1}
2^{-j(d-1)/2}\prod_{i=1}^{d+1} \(|x_iy_i|+ {2^{j}}{n^{-2}}\)^{-\k_i}.
$$
\end{cor}

Recall that $c(\varpi, \t)$ is a fixed spherical cap, $\ta\in
[n^{-1},\pi]$ and  $\ga = \ga_\varpi$ is defined in
\eqref{def:gamma}. We are now in a position to prove the following
$L^\infty$ estimate:

\begin{lem}\label{lem-2-4}
If $f$ is supported in $c(\varpi, \t)$, then
\begin{align*}
\sup_{x\in c(\varpi, \t)} \left |Y_{n,j}(f)(x)\right|
 \leq c n^{d-1+2\ga} 2^{-j(\f {d-1}2+\ga)} \t^{2\ga+d}\Bl[
\int_{c(\varpi, \t)}h_\k^2(x)\,
d\varpi(x)\Br]^{-1}\|f\|_{\k,1}.
\end{align*}
\end{lem}

\begin{proof}
Note that if $x \in c(\varpi,\t)$, then $|x_i-\varpi_i| \le
\|x-\varpi\|\le d (x,\varpi) \le \t$ so that $\f 34 |\varpi_i|\leq
|x_i|\leq \f 54  |\varpi_i|$ for $1\leq i\leq v$, and $|x_i|\leq
5\ta$ for $v+1\leq i \leq d+1$. It follows from Corollary
\ref{lem-2-3} that, for any $x, y\in c(\varpi,\t)$,
\begin{align*}
\Bl|V_\k\Bl[ C_{n,j}(\la x,\cdot\ra)\Br](y)\Br|
&\leq c  n^{d-1} 2^{-j(d-1)/2} \prod_{i=1}^{v}
    |\varpi_i|^{-2\k_i} \prod_{i=v+1}^{d+1} n^{2\k_i} {2^{-j\k_i}}\\
& \leq c n^{d-1} 2^{-j(\f {d-1}2+\ga)}(n\t)^{2\ga}
    \prod_{i=1}^{d+1} (|\varpi_i|+\t)^{-2\k_i}\\
&\leq c n^{d-1} 2^{-j(\f {d-1}2+\ga)}(n\t)^{2\ga} \t^d
    \Bl[\int_{c(\varpi,\t)}h_\k^2(z)\, d\og(z)\Br]^{-1},
\end{align*}
where the last step follows from the relation \eqref{eq:doubling}.
This implies that
\begin{align*}
  \sup_{x\in c(\varpi, \t)} |Y_{n,j}(f)(x)|
&\leq  \sup_{x\in c(\varpi, \t)} \int_{c(\varpi,\t)} |f(y)|\Bl|
V_{\k} \Bl[ C_{n,j}\la x, \cdot\ra\Br](y)\Br|h_\k^2(y)\,
d\og(y)\\
& \leq c n^{d-1+2\ga} 2^{-j\(\f {d-1}2+\ga\)} \t^{2\ga+d}\Bl[
\int_{c(\varpi, \t)}h_\k^2(x)\, d\og(x)\Br]^{-1}\|f\|_{\k,1},
\end{align*}
which is the desired inequality.
\end{proof}

\subsubsection{$L^2$ estimates} We prove the following estimate:

\begin{lem}\label{lem-2-5}
For any $f\in L^2(h_\k^2; S^d)$,
\begin{equation*}
        \|Y_{n,j} (f)\|_{\k,2}\leq c\,{n^{-1}} {2^{j}} \|f\|_{\k,2}.
\end{equation*}
\end{lem}

\begin{proof}
For simplicity, we shall write $\xi_j=\xi_1$ for $j\ge 1$. Also let
$\l = \l_\k$ in this proof.
From \eqref{projf*g} and the definition of $Y_{n,j}$ in \eqref{2-2}, it
follows that  each $Y_{n,j}$ is  a multiplier operator,
$$
Y_{n,j}(f) = \sum_{k=0}^\infty m_{n,j}(k) \proj_k(h_\k^2; f),
$$
where the equality is understood in a distributional sense, and
$$
m_{n,j}(k) := c_{n,k} \int_0^{\pi} C_n^{\l}(\cos t) C_k^{\l}(\cos
t)
          \xi_j\(\f{n^2\sin^2t}{4^{j-1}}\)\sin^{2\l}t\, dt
$$
with $|c_{n,k}|\leq c  n k^{-2\l+1}$. Hence, it is enough to prove
\begin{equation}\label{2-6}
          \sup_k |m_{n,j}(k)|\leq c \, n^{-1} 2^{j}.
\end{equation}

If  $k\ge \f n4$, then using the fact that $|(\sin \t)^\l C_n^\l(\cos \t)| \le c n^{\l-1}$,
a straightforward computation gives
\begin{align*}
|m_{n,j}(k)|&\leq |c_{n,k}| \int_0^{\pi} \Bl|C_n^{\l}(\cos t)
C_k^{\l}(\cos t)
\xi_j\(\f{n^2\sin^2t}{4^{j-1}}\)\Br|\sin^{2\l} t\, d t\\
&\leq c\int_0^\pi \Bl|\xi_j\(\f{n^2\sin^2 t}{4^{j-1}}\)\Br|\,
dt\leq c \f {2^j}n,
\end{align*}
where the last step follows easily using the support of $\xi_j$.

For  $k\leq \f n4$,  we shall use the following  formula
(cf. \cite[p. 319, Theorem 6.8.2]{AAR}),
\begin{equation}\label{2-7}
 C_k^{\ld}(t) C_n^{\ld} (t)=\sum_{i=0}^{\min\{k,
n\}} a(i, k, n) C_{k+n-2i}^\ld (t),
\end{equation}
where
$$
a(i,k,n):= \f{ (k+n+\ld-2i)(\ld)_i (\ld)_{k-i} (\ld)_{n-i}
(2\ld)_{k+n-i}}{ (k+n+\ld-i) i! (k-i)! (n-i)! (\ld)_{k+n-i}}\f{
(k+n-2i)!}{(2\ld)_{k+n-2i}}.
$$
For $k \le n/4$, it is easy to see that
\begin{align} |a(i, k, n)| &\sim
 \(\f{(i+1)(\min\{k,n\}-i+1)  (k+n-i+1)}{k+n-2i+1}\) ^{\ld-1} \notag\\
  & \sim (i+1)^{\ld-1} (k-i+1)^{\ld-1}.\label{4-10-v2}
\end{align}
Consequently, it follows that for $k\leq n/4$,
\begin{align*}
    |m_{n,j}(k)|&\leq c  n k^{-2\l+1}\sum_{i=0}^k (i+1)^{\l-1}\times \\
 & \qquad \times (k-i+1)^{\l-1} \Bl|\int_0^{\pi} C_{k+n-2i}^{\l}(\cos t)
     \xi_j\(\f{n^2\sin^2t}{4^{j-1}}\)\sin^{2\l}t\, dt\Br|\\
&  \leq c n^{\l+\f12} \max_{ 3n/4 \leq m\leq 5n/4}\Bl|\int_{-1}^{1}
  P_{m}^{(\l-\f12,\l-\f12)}(s)\xi_j\(\f{n^2(1-s^2)}
  {4^{j-1}}\)(1-s^2)^{\l-\f12}\, ds\Br|.
\end{align*}
Then using the estimate \eqref{Est-Jacobi} we obtain
$$
   m_{n,0}(k)\leq c n^{2\l} \int_{1-|s|\leq c n^{-2}}(1-|s|)^{\l-\f12}\,
        ds\leq c n^{-1}.
$$
If $j\ge 1$, then for all $\ell \in \NN$, it follows that
$$
\Bl|\f{d^\ell}{ds^\ell}\( \xi_1\(\f
{n^2(1-s^2)}{4^{j-1}}\)(1-s^2)^{\l-\f12}\)\Br|
 \leq c \(\f {2^{j}}{n}\)^{2\l-1-2\ell},
$$
since $1-s^2 \sim (\f{2^j}n)^2$ in  the support of $\xi_1'$;
consequently, we obtain by integration by parts, \eqref{d-Jacobi}
and \eqref{Est-Jacobi} that
\begin{align*}
& m_{n,j}(k)\leq c n^{\l+\f12-\ell}\times\\
& \quad \times  \max_{ 3n/4\leq m\leq 5n/4}\Bl|\int_{-1}^{1}
  P_{m+\ell}^{(\l-\f12-\ell,\l-\f12-\ell)}(s) \f{d^\ell}{ds^\ell}\(
  \xi_1\(\f {n^2(1-s^2)}{4^{j-1}}\)(1-s^2)^{\l-\f12}\)\, ds\Br|\\
&\leq c 2^{j(\l-\ell)}{2^j} n^{-1} \le c 2^j n^{-1}
\end{align*}
upon choosing $\ell \ge \l$, Thus, in both cases, we get the desired estimate.
\end{proof}

\subsubsection{Proof of Theorem \ref{thm:proj-local},  case I: $\g < \s_\k -\frac{d-1}{2}$}
Recall $\nu=\f{2+2\sa_\k}{2+\sa_\k}$. We set, in this subsection,
$$
     A: = \int_{c(\varpi,\t)}h_\k^2(y)\, d\og(y).
$$

Recall the decomposition \eqref{2-2}. For a generic $f$, we set
$$
T_{n,j} f : =Y_{n,j}(f\chi_{c(\varpi,\t)})\chi_{c(\varpi,\t)}, \qquad 0\leq j\leq L_n.
$$
Clearly,  if $f$ is supported in $c(\varpi, \t)$  and  $x\in
c(\varpi,\t)$, then $T_{n,j} f(x) = Y_{n,j} f(x)$. Using  Lemmas
\ref{lem-2-4} and \ref{lem-2-5}, we have
\begin{align} \label{2-8}
\|T_{n,j}f\|_\infty &\leq c n^{2\ga+d-1}
2^{-j(\f{d-1}2+\ga)}\t^{2\ga+d} A^{-1} \|f \chi_{c(\varpi,\t)}\|_{\k,1}, \\
\|T_{n,j}f\|_{\k,2}&\leq c \, n^{-1} {2^j} \|f \chi_{c(\varpi,\t)}\|_{\k,2}. \notag
\end{align}
Hence, by the Riesz-Thorin convexity theorem, we obtain
\begin{align} \label{2-9}
 \|T_{n,j}f\|_{\k, \nu'} &\leq c n^{-1}
2^{j\(1-(\f{d+1}2+\ga)\f1{\sa_\k+1}\)}(n\t)^{\f{2\ga+d}{\sa_\k+1}}
A^{1-\f2{\nu}}\|f\|_{\k,\nu}.
\end{align}
On the other hand, using (\ref{2-8}), H\"older's inequality and
\eqref{eq:doubling}, we obtain
\begin{align}
  \|T_{n,j}f \|_{\k, \nu'} & \leq
     \|T_{n,j}f \|_{\infty}A^{1-\f1\nu}\leq c n^{2\ga+d-1}
     2^{-j\(\f{d-1}2+\ga\)}\t^{2\ga+d} A^{-\f1\nu}
       \|f\chi_{c(\varpi,\t)}\|_{\k,1}\notag\\
&\leq c n^{-1} 2^{-j\(\f{d-1}2+\ga\)}(n \t)^{2\ga+d} A^{1-\f2\nu}
\|f\|_{\k,\nu}. \label{2-10}
\end{align}

Now assume that $f$ is supported in $c(\varpi,\t)$ and $\f
{2^{j_0-1}}n\leq \ta\leq \f {2^{j_0}}n$ for some $1\leq j_0\leq
L_n$. Using (\ref{2-2}) and Minkowski's inequality, we have
\begin{align*}
 \left \|\proj_n(h_\k^2; f) \chi_{c(\varpi,\t)} \right\|_{\k,\nu'}
 \leq \sum_{j=0}^{2j_0} \|T_{n,j}f\|_{\k, \nu'}+
\sum_{j=2j_0+1}^{L_n} \|T_{n,j}f\|_{\k, \nu'} = :  \Sigma_1+\Sigma_2.
\end{align*}
For the first sum $\Sigma_1$, we use (\ref{2-9}) to obtain
\begin{align*}
 \Sigma_1 &\leq c n^{-1}(n\t)^{\f{2\ga+d}{\sa_\k+1}} A^{1-\f 2\nu}\|f\|_{\k,\nu}
  \sum_{j=0}^{2j_0} 2^{j\(1-(\f{d+1}2+\ga)\f1{\sa_\k+1}\)}\\
& \leq c n^{\f{\sa_\k}{1+\sa_\k}}\t^{\f{2\sa_\k+1}{\sa_\k+1}}A^{1-\f 2\nu}
  \|f\|_{\k,\nu},
\end{align*}
since  $\g < \s_\k- \frac{d-1}{2}$ readily implies that
$1-(\f{d+1}2+\ga)\f1{\sa_\k+1} > 0$. For the second sum
$\Sigma_2$, we use (\ref{2-10}) to obtain
\begin{align*}
\Sigma_2&\leq c n^{-1}(n\t)^{2\ga+d} A^{1-\f2\nu}
    \|f\|_{\k,\nu} \sum_{j=2j_0+1}^{\infty} 2^{-j(\f{d-1}2+\ga)} \\
&  \leq cn^{-1}(n\t) A^{1-\f2\nu} \|f\|_{\k,\nu} \leq
    c n^{-1}(n\t)^{\f{2\sa_\k+1}{\sa_\k+1}} A^{1-\f2\nu} \|f\|_{\k,\nu}\\
&= c n^{\f{\sa_\k}{1+\sa_\k}}\t^{\f{2\sa_\k+1}{\sa_\k+1}}A^{1-\f2\nu}
     \|f\|_{\k,\nu},
\end{align*}
where in the third inequality we have used the fact that $n\t \ge 1$.

Putting the above together proves Theorem \ref{thm:proj-local} in the
case $\ga<\sa_\k-\f{d-1}2$. \qed


\subsection{Proof of Theorem \ref{thm:proj-local},  case 2:
$\g = \s_\k -\frac{d-1}{2}$}

Recall that $|\varpi_j|\ge 4\ta$ for $1\leq j\leq v$, $|\varpi_j|<
4\ta$ for $v+1\leq j\leq d+1$, and
 $\g = \g_\varpi=\sum_{j=v+1}^{d+1}\k_j$.
In this case, either    $v=1$ and $|\varpi_1|=\max_{1\leq j\leq
d+1} |\varpi_j|\ge \f 1{\sqrt{d+1}}$;  or  $v\ge 2$ and
$\k_1=\cdots=\k_{v}=0$.
 Therefore, by \eqref{eq:doubling}, we have
$$
\int_{c(\varpi, \t)}h_\k^2(x)\, d\o(x) \sim
\ta^d\bigl(\prod_{j=1}^v |\varpi_j|^{2\k_j}\bigr) \ta^{2\ga}\sim
\t^{2\sa_\k +1}.
$$
Hence, Theorem \ref{thm:proj-local} in this case is equivalent to
the following proposition:

\begin{prop}\label{prop-2-6}
Let  $f$ be supported in $c(\varpi, \t)$ with $\t \in(n^{-1}, 1/(8d)]$ and let
$\nu: =\f{2\sa_\k +2}{\sa_\k+2}$ and $\nu':=\f{\nu}{\nu-1}$. Then
$$
   \left \|\proj_n(h_\k^2;f)\chi_{c(\varpi,\t)} \right \|_{\k,\nu'}
      \leq c n^{\f {\sa_\k}{(1+\sa_\k)}} \|f\|_{\k, \nu}.
$$
\end{prop}

To prove Proposition \ref{prop-2-6}, we use the method of analytic
interpolation \cite{St}. For $z\in \mathbb{C}$, define
\begin{equation}\label{4-14-v2}
  \mathcal{P}_n^z f (x): = (f\ast_\k G_n^z)(x) =
      a_\k\int_{S^d} f(y) V_\k \Bl[ G^z_n (\la x, \cdot\ra)\Br](y)
       h_\k^2 (y) \, d\og(y)
\end{equation}
for $x\in S^d$, where
\begin{align}
    G_n^z(t) &= (\sa_\k +1)(1-z) \f{n+\l_\k}{\l_\k} C_n^{\l_\k} (t)
           (1-t^2+n^{-2})^{ \f {\sa_\k-(\sa_\k+1)z}2}.\label{4-15-v2}
\end{align}
From \eqref{projection}, it readily follows that
$$
     \mathcal{P}_n^{\f{\sa_\k}{1+\sa_\k}}f=\proj_n(h_\k^2;f).
$$
For the rest of this subsection, we shall use $c_\tau$ to denote a
general constant satisfying $|c_\tau| \leq c (1+|\tau|)^\ell$ for
some inessential positive number $\ell$.

\subsubsection{Estimate fot $z = 1 + \i \tau$}

\begin{lem}\label{lem-2-7}
For $\tau \in \mathbb{R}$,
$$
   \|\mathcal{P}_n^{1+ \i \tau} f\|_{\k, 2 }\le c_\tau  \|f\|_{\k, 2}.
$$
\end{lem}

\begin{proof}
From \eqref{4-14-v2}, \eqref{4-15-v2} and \eqref{projf*g}, it
follows that
$$
 \proj_k (h_\k^2; \mathcal{P}_n^{1+\i\tau} f) =  J_n(k)  \proj_k(h_\k^2; f),
    \qquad k=0,1,\cdots,
$$
where
\begin{equation*}
J_{n}(k):=  \CO(1) n k^{-2\l_\k+1} \tau' \int_0^\pi
 C_k^{\l_\k}(\cos t) C_n^{\l_\k}(\cos t)
 (\sin^2 t+n^{-2})^{-\f12+i\tau'}(\sin t)^{2\l_\k}\, dt
\end{equation*}
and $ \tau'=-\f{\sa_\k+1}2\tau$. Therefore, it is sufficient to prove
\begin{equation}\label{2-12}
        |J_n(k)| \leq c_{\tau}, \qquad   \forall k,n\in\mathbb{N}.
\end{equation}

For $k<\f n4$, (\ref{2-12}) can be shown as in the proof of Lemma
 \ref{2-5}. In fact,   using  (\ref{2-7})and  (\ref{4-10-v2}),   we  obtain
\begin{align*}
 |J_n(k)| &\leq c |\tau| n^{\l_\k+\f12}  \\
  & \times \max_{3n/4\leq m \leq
5n/4}\Bl|\int_{-1}^1 P_m^{(\l_\k-\f12,\l_\k-\f12)} (s) (1-s^2
+n^{-2})^{-\f12+i\tau'}(1-s^2)^{\l_\k-\f12}\, ds \Br|,
\end{align*}
which is controlled  by
\begin{align*}
&c_\tau+ c_\tau n^{\l_\k+\f12-\ell} \max_{3n/4\leq m \leq
5n/4}\int_{-1+n^{-2}}^{1-n^{-2}}
\Bl|P_{m+\ell}^{(\l_\k-\f12-\ell,\l_\k-\f12-\ell)}
(s)\times \\
&\hspace{1cm}\times \f{d^\ell}{ds^\ell}\( (1-s^2
+n^{-2})^{-\f12+i\tau'}(1-s^2)^{\l_\k-\f12}\)\Br|\, ds\\
&\leq c_\tau'
\end{align*}
using integration by parts $\ell>\l_\k$ times. This proves (\ref{2-12}) for
$k<\f n4$.

For $k\ge \f n4$, (\ref{2-12}) can be established exactly as in
\cite[p. 54-55]{So1} (see also \cite[p.76--81]{WL}). For
completeness, we sketch the proof as follows. Since $C_j^{\l}(t)
= \CO(1) j^{\l-\f12}P_j^{(\l-\f12,\l-\f12)}(t)$ and
$P_j^{(\al,\b)}(-t)=(-1)^j P_j^{(\al,\be)}(t)$, we can write
\begin{align*}
J_{n}(k)&= \CO(1)k^{-\l_\k+\f12}n^{\l_\k+\f12}
 \tau'\Bl[ \int_0^{4n^{-1}} +\int_{4n^{-1}} ^{\f\pi2}\Br]
 P_k^{(\l_\k-\f12, \l_\k-\f12)}(\cos t) \\
  \   \   \  &\times P_n^{(\l_\k-\f12, \l_\k-\f12)}(\cos t)
 (\sin^2 t+n^{-2})^{-\f12+i\tau'}(\sin t)^{2\l_\k}\, dt\\
 &=: J_{n,1}(k)+J_{n,2}(k).
 \end{align*}
 Since  $|P_j^{(\al,\al)}(t)|\leq c j^\al$, a straightforward
 calculation shows $|J_{n,1}(k)|\leq c_\tau$. To estimate
 $J_{n,2}(k)$, we need  the asymptotics of the Jacobi polynomials as given in
\cite[p. 198]{Szego},
$$
 P_j^{(\alpha,\beta)}(\cos t) =
  \pi^{-\frac12} j^{-\frac12}   (\sin \tfrac{t}2 )^{-\a-\frac12}
  (\cos \tfrac{t}2)^{-\b-\frac12}
        \left[ \cos (N_j t+\tau_\al) + \CO(1) (j\sin t)^{-1} \right]
$$
for $ j^{-1} \le t \le \pi -  j^{-1}$, where $N_j=
j+\frac{\alpha+\beta+1}2$ and $\tau_\al = -\frac{\pi}2
({\alpha+\frac12}).$ Applying this asymptotic formula with $\a
=\b= \l_\k-1/2$,   we obtain, for $k\ge \f n4$ and $4n^{-1}\leq t
\leq \f \pi2$,
\begin{align*}
&k^{-\l_\k+\f12}n^{\l_\k+\f12}P_k^{(\l_\k-\f12, \l_\k-\f12)}(\cos
t)  P_n^{(\l_\k-\f12, \l_\k-\f12)}(\cos t)
 (\sin t)^{2\l_\k}\\
&=\CO(1)\Bigl[ \cos \bigl(( k-n)t\bigr)+\cos \bigl((k+n+2\l_\k)t
-\l_\k \pi\bigr)\Bigr]+ \CO\( \f 1{nt}\)
\end{align*}
using the cosine addition formula. Also, note that
$$
\Bigl( \f 1 {\sin^2 t+n^{-2}}\Bigr)^{\f 12 -i\tau'} =
t^{-1+2i\tau'} + \CO(t) + \CO \(n^{-2}t^{-3}\),\   \   \ 4n^{-1}\leq t
\leq \f \pi2.
$$
It follows that
\begin{align*}
|J_{n,2}(k)|&\leq c_{\tau} + c_{\tau}
\sup_{\ell\in\mathbb{R}}\Bl|2\tau'
\int_{4 n^{-1}}^{\f\pi2} t^{-1+2i\tau'} e^{i\ell t}\, dt \Br|\\
&\leq c_{\tau} + c_{\tau} \sup_{a<b} \Bl|\int_{a}^b  e^{i t}\,
dt^{2i\tau'} \Br|\leq c_\tau.
\end{align*}
This
  proves the desired inequality (\ref{2-12}) for
$k\ge \f n4$.
\end{proof}

\subsubsection{Estimate fot $z = \i \tau$}

\begin{lem}\label{lem-2-8}
If $\tau \in \mathbb{R}$ and $f$ is supported in $c(\varpi,\t)$
then
\begin{equation*}
\sup_{x\in c(\varpi,\t)} | \mathcal{P}_n^{\i \tau} f(x)| \le
c_\tau n^{ \sa_\k} \|f\|_{1,\k}.
\end{equation*}
\end{lem}

\begin{proof}
Since $f$ is supported in $c(\varpi,\t)$, we have
\begin{align*}
\sup_{x\in c(\varpi,\t)}|\mathcal{P}_n^{\i\tau} (f)(x)|& \leq
\sup_{x\in c(\varpi,\t)}\int_{c(\varpi,\t)} | f(y)|\left| V_\k
\left[
      G^{\i\tau}_n (\la x,\cdot\ra)\right](y)\right| h_\k^2 (y) \, d\og(y)\\
&\leq \|f\|_{1,\k} \sup_{x, y \in c(\varpi, \t)}\left| V_\k \[
    G^{\i\tau}_n (\la x, \cdot\ra)\](y)\right|.
\end{align*}
Thus, it is sufficient to prove
\begin{equation}\label{2-13}
\left | V_\k \[ G^{\i\tau}_n (\la x, \cdot\ra)\](y)\right|\leq
 c_\tau n^{ \sa_\k}\qquad   \text{for all $x, y\in c(\varpi,\t)$}.
 \end{equation}
 We note that \eqref{2-13} is trivial when $\k_{\rm min}=0$ since
 in this case $\|G^{i\tau}_n\|_\infty \leq c_\tau n^{\l_\k}=c_\tau
 n^{\sa_\k}$. So we shall assume $\k_{\rm min}>0$ for the rest of
 the proof.

To prove  (\ref{2-13}), we claim that it's enough to prove that
\begin{equation}\label{2-14}
\Bl|\int_{-1}^1 G^{\i\tau}_n (at+s)(1-t^2)^{\da-1}(1+t)\, dt \Br|
 \leq c_{\tau} n^{\sa_\k},
\end{equation}
whenever  $|a|\ge \va_d > 0$, $|a|+|s|\leq 1$,  $\da\ge \k_{\rm
min}$, where $c_\tau$ is independent of  $s$.

To see this, let $x, y\in c(\varpi,\t)$ and without loss of
generality, assume $\varpi_1=\max_{1\leq j\leq d+1}|\varpi_j|$.
Then   $\varpi_1 \ge 1/\sqrt{1+d}$, which implies that $|x_1 |,
|y_1| \ge 1/\sqrt{d+1} - \t \ge 1/\sqrt{d+1}-1/(8d)>0$, so that
$|x_1y_1| \ge \va_d >0$. Thus, invoking  (\ref{2-14}) with $a=
x_1y_1$, $\da=\k_1$ and $s=\sum_{j=2}^{d+1} t_j x_jy_j$ gives
$$
\left|\int_{-1}^1 G^{\i\tau}_n \Bl(\sum_{j=1}^{d+1}x_jy_jt_j\Br)
       (1-t_1^2)^{\k_1 -1} (1+t_1)\, dt_1 \right|\leq c_{\tau} n^{\sa_\k}.
$$
The desired inequality (\ref{2-13}) then follows by the Fubini theorem and the
integral representation of $V_\k$ in \eqref{eq:Vk}. This proves the claim.

For the proof of (\ref{2-14}), by symmetry, it is sufficient to prove
\begin{equation}\label{2-15}
\left |\int_{-1}^1 G^{\i\tau}_n (at+s)(1-t)^{\da-1}\xi(t)\, dt \right|
     \leq c_{\tau} n^{\sa_\k},
\end{equation}
where $\xi$ is a $C^\infty$ function supported in $[-\f12,1]$,
whenever $|a|\ge \va_d>0$, $|a|+|s|\leq 1$ and $\da\ge \k_{\rm
min}$.

Let $\eta_0\in C^\infty(\mathbb{R})$ be such that $\chi_{[-\f12,\f12]}\leq \eta_0
\leq \chi_{[-1,1]}$, and let $\eta_1(t):=1-\eta_0(t)$.
Set, in this subsection,
$$
     B:= \f{ n^{-1} +\sqrt{1-|a+s|}} {4n}.
$$
We then split the integral in (\ref{2-15}) into a sum $I_0(a,s)+I_1(a,s)$ with
\begin{align*}
I_j(a,s) :=\int_{-1}^1  G^{\i\tau}_n (at+s)\eta_j
    \Bl(\f{1-t}{B}\Br)(1-t)^{\da-1}\xi(t)\, dt,  \qquad    j=0,1.
\end{align*}

It is easy to verify that $1 +n \sqrt{ 1-|at+s|} \sim 1+n \sqrt{1-|a+s|}$ whenever
$t\in [1-B, 1]\cap [-1,1]$. Therefore,   for $1-B\leq t\leq 1$,  using
\eqref{Est-Jacobi},
\begin{align*}
 |G_n^{\i\tau}(at+s)| \leq c n^{\l_\k} (n^{-1}+\sqrt{1-|at+s|})^{-\l_\k+\sa_\k}
     \leq c\,n^{\sa_\k}  B^{-\k_{\rm min}},
\end{align*}
which implies that
\begin{align*}
|I_0(a, s)| \leq   c \int_{\max\{1-B, -\f12\}}^1 |G_n^{\i\tau}(at+s)
    | (1-t)^{\da-1}\, dt\leq c n^{\sa_\k}  B^{\da- \k_{\rm min}} \leq c n^{\sa_\k}.
\end{align*}

To estimate $I_1(a,s)$, we write
$$
G^{\i\tau}_n (at+s)\eta_1\Bl(\f{1-t}{B}\Br)(1-t)^{\da-1}\xi(t)=c_n
   P_n^{(\l_\k-\f12,\l_\k-\f12)}(at+s)\vi(t),
$$ where $|c_n|\leq c_\tau n^{\l_\k+\f12}$ and
 \begin{equation*}
   \vi(t): = \(1-(at+s)^2+n^{-2}\)^{\f{\sa_\k}2-\f{\sa_\k+1}2\i\tau}
     \eta_1\Bl(\f{1-t}B\Br)\xi(t)(1-t)^{\da-1}.
\end{equation*}
Recall $|a|\ge \va_d>0$.  Using integration by parts $\ell$ times
gives
$$
|I_1(a,s)|\leq c n^{\l_\k+\f12-\ell}\int_{-1}^1
 \left|P_{n+\ell}^{(\l_\k-\f12-\ell,\l_\k-\f12-\ell)}(at+s)\right| |\vi^{(\ell)}(t)|\, dt.
$$
If $-\f12\leq t \leq 1-B/2$, then
$
1-|at+s|\ge 1-|a|-|s|+(1-|t|)|a|\ge c(1-t)\ge c B\ge c n^{-2}
$
which  implies, in particular,
$\(1-(at+s)^2+n^{-2}\)^{-1}\leq c (1-t)^{-1}$.
Since $\vi$ is supported in $(-\f12, 1-\f B2)$, which gives $B^{-1} \le (1-t)^{-1}$,
it follows from Lebnitz' rule that
$$
|\vi^{(\ell)}(t)|\leq c_{\tau} (1-|at+s|)^{\f{\sa_\k}2}(1-t)^{\da-\ell-1}.
$$
Therefore, choosing $\ell>2\da$ and recalling that $\d \ge \k_{\rm min}$,
we have by \eqref{Est-Jacobi}  that
\begin{align*}
|I_1(a,s)|&\leq c_{\tau}n^{\l_\k-\ell}\int_{-\f12}^{1-\f B2}(1-|at+s|)^{\f
  {\sa_\k+\ell}2-\f{\l_\k}2}(1-t)^{\da-1-\ell}\, dt\\
& \leq c_{\tau} n^{\l_\k-\ell}\int_{\f {B|a|}2}^{\f 32|a|}
      (1-|a+s|+u)^{\f{\sa_\k+\ell}2-\f{\l_\k} 2}u^{\da-1-\ell}\,du.
\end{align*}
Using the fact that $(1-|a+s|+u)^\a \le c ( (1-|a+s|)^\a + u^\a )$ we break
the last integral into a sum $J_1+J_2$, where
\begin{align*}
J_1 & \le c_{\tau} n^{\l_\k-\ell}\int_{\f {B|a|}2}^{\infty}
          (1-|a+s|)^{\f{\sa_\k+\ell}2-\f{\l_\k}2}  u^{\da-1-\ell}\,du \\
       & \le c_\tau n^{\l_\k-\ell}(1-|a+s|)^{\f{\sa_\k+\ell}2-\f{\l_\k}2} B^{\d-\ell}
          \le c_\tau n^{\l_\k-\ell} n^{\ell-\d} (nB)^{\d +\s_\k - \l_\k} \\
       & = c_\tau n^{\l_\k-\d}(nB)^{\d - \k_{\rm min}} \le c_\tau n^{\l_\k -\d}
        \le c_\tau n^{\s_k},
\end{align*}
and
\begin{align*}
 J_2 & \le c_{\tau} n^{\l_\k-\ell}\int_{\f {B|a|}2}^{\infty}u^{\f{\sa+\ell}2-\f{\l_\k}2}
     u^{\da-1-\ell}\,  du
         \le  c_{\tau} n^{\l_\k-\ell} B^{\f{\sa_\k-\ell}2-\f{\l_\k}2 +\d} \\
        &   =  c_{\tau} n^{\l_\k-\da} (n^2 B)^{\f{\k_{\rm min}-\ell}2}
                     (n B)^{\d - \k_{\rm min}}
         \le  c_{\tau} n^{\l_\k-\d} \le c_\tau n^{\s_\k}.
\end{align*}

Putting the above together, we obtain the desired estimate
(\ref{2-15}) and  complete the proof of Lemma \ref{lem-2-8}.
\end{proof}

\subsubsection{Proof of Proposition \ref{prop-2-6}}
Define
$$
T^z f=
n^{\sa_\k(z-1)}\mathcal{P}_n^z(f\chi_{c(\varpi,\t)})\chi_{c(\varpi,\t)},
\quad
   0 \leq \Re\, z \leq 1.
$$
By Lemmas \ref{lem-2-8} and \ref{lem-2-7},
we have
\begin{align*}
 \|T^{1+ \i \tau} f\|_{\k, 2 }\le c_\tau  \|f\|_{\k, 2} \quad\hbox{and}\quad
\| T^{\i \tau} f\|_\infty \le c_\tau \|f\|_{1,\k}.
\end{align*}
These allow us to apply Stein's interpolation theorem \cite[p. 205]{St}
to the analytic family of operators $T^z$, which yields
$$
\|T^{\f{\sa_\k}{1+\sa_\k}}f\|_{\k,\nu'} \leq c \|f\|_{\k, \nu}.
$$
Consequently, using the fact that
$$
T^{\f{\sa_\k}{1+\sa_\k}}f=n^{-\f{\sa_\k}{1+\sa_\k}}\mathcal{P}_n^{\f{\sa_\k}{1+\sa_\k}}
  (f\chi_{c(\varpi,\t)})\chi_{c(\varpi,\t)}  =
  n^{-\f{\sa_\k}{1+\sa_\k}}\proj_n (h_\k^2;f\chi_{c(\varpi,\t)})\chi_{c(\varpi,\t)},
$$
we have proved Proposition \ref{prop-2-6}. \qed


\section{Boundedness of projection operator}
\setcounter{equation}{0}

The objective of this section is to prove Theorems \ref{thm:proj} and
\ref{thm:proj-cap}.

\subsection{ Proof of Theorem \ref{thm:proj-cap}}
Assume that $f$ is supported in a spherical cap $c(\varpi, \t)$.  Without
loss of generality, we may assume $\t< 1/(8d)$, since otherwise we
can decompose $f$ as a finite sum of functions supported on a family
of spherical caps of radius $< 1/(8d)$.

We start with the case $p=1$. By the definition of the projection operator,
it follows from the integral version of the Minkowski inequality and
orthogonality that
\begin{align*}
 \|\proj_n(h_\k^2; f)\|_{\k,2} & \leq \sup_{y\in c(\varpi,\t)}
    \(\int_{S^d} \left |P_n(h_\k^2;x,y) \right|^2
         h_\k^2(x)\, d\og(x)\)^{1/2}\|f\|_{\k,1}\\
& =  \Bl(\sup_{y\in c(\varpi, \t)} P_n(h_\k^2; y,y) \Br)^{1/2}\|f\|_{\k,1}.
\end{align*}
Using the pointwise estimate of the kernel in \eqref{eq:S-est} and the
fact that $n\t \ge 1$,  we then obtain
\begin{align*}
 \|\proj_n(h_\k^2; f)\|_{\k,2} &  \leq cn^{\f{d-1}2}\sup_{y\in c(\varpi, \t)}
       \prod_{j=1}^{d+1}(|y_j|+n^{-1})^{-\kappa_j}\|f\|_{\k,1}\\
& \leq c n^{\f{d-1}2}(n\t)^{\sa_\k-\f{d-1}2}\sup_{y\in c(\varpi, \t)}
    \prod_{j=1}^{d+1}(|y_j|+\t)^{-\kappa_j}\|f\|_{\k,1} \\
&\leq c n^{\sa_\k} \t^{\sa_\k+\f12} \Bl(\int_{c(\varpi,\t)}h_\k^2(y)\,
d\og(y)\Br)^{-\f12}   \|f\|_{\k,1},
 \end{align*}
where the last step follows from \eqref{eq:doubling}. This proves
Theorem \ref{thm:proj-cap} for $p=1$.

Next, we use  H\"older's inequality and Theorem \ref{thm:proj-local} to obtain
\begin{align*}
\|\proj_n(h_\k^2;f)\|_{\k,2}^2 & =
\int_{c(\varpi,\t)} f(y) \proj_n(h_\k^2;f)(y)h_\k^2(y)\,
d\og(y) \\
& \leq \|f\|_{\k,\nu}\Bl(\int_{c(\varpi, \t)}|\proj_n(h_\k^2;f, y)|^{\nu'}
 h_\k^2(y)\, d\og(y)\Br)^{\f1{\nu'}} \\
& \leq c n^{\f{\sa_\k}{\sa_\k+1}}
  \t^{\f{2\sa_\k+1}{\sa_\k+1}} \Bl[  \int_{c(\varpi, \t)}
   h_\k^2(x)\, d\sa(x)\Br]^{1 -\f2\nu}\|f\|^2_{\k,\nu},
\end{align*}
which proves  Theorem \ref{thm:proj-cap} for
$p=\nu=\f{2\sa_\k+2}{\sa_\k+2}$.

Finally, Theorem \ref{thm:proj-cap} for $1\leq p\leq \nu$ follows by
applying  the Riesz-Thorin convexity theorem to the linear operator
$g\mapsto \proj_n(h_\k^2; g\chi_{c(\varpi,\t)})$.
\qed

\subsection{ Proof of Theorem \ref{thm:proj}}

Theorem \ref{thm:proj} (i) follows directly by invoking Theorem
\ref{thm:proj-local} with $\ta=\pi$.  Theorem \ref{thm:proj} (ii)
follows from the Riesz-Thorin convexity theorem applied to the
boundedness of $f \mapsto \proj_n(h_\k^2; f)$ in (2,2) and in
$(\nu,2)$.

We now prove that the estimates are sharp. We start with a duality result
whose proof is standard:

\begin{lem} \label{duality}
Assume $1 \le p \leq  2 \leq  q \le \infty$ and $\frac1p +\frac1q
=1$. Then the followings are equivalent:
\begin{enumerate}
 \item[(i)] $\|\proj_n(h_\k^2; f)\|_{k,2} \le A \|f\|_{k,p}$,
\item[(ii)]  $\|\proj_n(h_\k^2; f)\|_{k,q} \le A \|f\|_{k,2}$.
\end{enumerate}
\end{lem}

To prove the sharpness of the estimates, we can assume without
loss of generality that  $\k_{\rm min} = \k_1$. For the case
Theorem \ref{thm:proj} (i) we define
$$
     f_{n}(x) : = P_n(h_\k^2;x,e), \qquad e = (1,0,0,\ldots,0).
$$
Since  $f_n \in \CH_n^{d+1}(h_\k^2)$, we have
$$
  \|\proj_n(h_\k^2;f_{n})\|_{\k,q} = \|f_{n}\|_{\k,q}  = \left ( \int_{S^d}
          |P_n(h_\k^2; x, e)|^q h_\k^2(x)d\o(x)   \right)^{1/q}.
$$
Thus, it is sufficient to show that
\begin{equation} \label{f-jn}
 \|f_{n}\|_{k,q} \sim n^{\s_k - \frac{2\s_\k+1}{q}} \|f\|_{k,2} \qquad
 \hbox{for $q \ge \frac{2 (\s_\k+1)}{\s_\k}$ }.
\end{equation}
Indeed, setting $p = q/(q-1)$ and using Lemma \ref{duality},
\eqref{f-jn} shows that
$$
 \|\proj_n(h_\k^2;f_n)\|_{\k,2}  \sim c\, n^{\s_k - \frac{2\s_\k+1}{q}}
  \|f_n\|_{\k,p}
        = c\, n^{(2\s_k+1)\left(\frac1p- \frac{\s_\k+1}{2\s_k+1} \right)}
        \|f_n\|_{\k,p},
$$
which proves the sharpness of (i).

Recall that $C_n^{(\l,\mu)} (t)$ denote the generalized Gegenbauer
polynomial. It is connected to $C_n^\l$ by an integral formula, which
implies by \eqref{proj-Kernel} that
$$
 P_n(h_\k^2;x,e) =  \frac{n+\l_\k}{\l_k} C_n^{(\s_\k,\k_1)} (x_1).
$$
Hence, using \eqref{G-Gegen}, in terms of Jacobi polynomials we have
\begin{equation}\label{5-2-v2}
P_{2n}(h_\k^2;x,e)  =  \CO(1) n^{\s_\k +\f12}
         P_n^{(\s_k-\frac12,\k_1-\frac12)}(2x_1^2-1).
\end{equation}
Since this is a function that only depends on $x_1$, a standard changing
variables leads to
\begin{align*}
 \|f_{2n}\|_{\k,q} & \sim n^{\s_\k+\f12}
     \(\int_{0}^{\pi} |P_n^{(\s_k-\frac12,\k_1-\frac12)}(2 \cos^2 \t-1)|^q
           |\cos \t|^{2\k_1} (\sin \t)^{2 \s_\k} d\t\)^{\f1q} \\
        & \sim n^{\s_\k+\f12} \left ( \int_{-1}^1
         |P_n^{(\s_\k-\frac12,\k_1-\frac12)}(t)|^q
        w^{(\s_k-\f12,\k_1-\f12)}(t)  d t \right)^{1/q}  \\
        & \sim n^{\s_\k} n^{\s_k- \frac{2 \s_k +1}{q}},
\end{align*}
where in the last step we have used \eqref{JacobiLp} and the
condition $q \ge 2(\s_k+1)/\s_\k >  (2 \s_k + 1 )/\s_\k$ to
conclude that the integral on $[0,1]$ has the stated estimate,
whereas the integral over $[-1,0]$, using $P_n^{(\a,\b)}(t) =
P_n^{(\b,\a)}(-t)$, has an order dominated by the integral on
$[0,1]$. For $q =2 $, using \eqref{5-2-v2}, we get
$$
  \|f_{2n} \|_{\k,2}= \(P_{2n}(h_\k^2; e,e)\)^{\f12}
  \sim n^{\s_\k}.
$$
 Together, these two relations establish \eqref{f-jn}
for even $n$. The proof of the odd $n$ is similar.
\qed

\begin{rem} \label{rem5.1}
For the ordinary spherical harmonics, the sharpness of part (ii) in
Theorem \ref{thm:proj} was proved in \cite{So1} with the help of the
function $(x_1+i x_2)^n$. For $h$-harmonics, it is then natural to consider
the function
$$
   F_n (x) : = V_\k [(x_1+i x_2)^n],  \qquad x \in \RR^d,
$$
where $V_\k$ is the intertwining operator associated with $h_\k^2$
and $\ZZ_2^d$. Since the Dunkl operator commutes with $V_\k$, so
is the $h$-Laplacian, which leads to $\Delta_h V_\k
\left[(x_1+ix_2)^n\right] = V_\k \left[ \Delta(x_1+ix_2)^n\right]
=0$, proving that $F_n(x)$ is an $h$-harmonic of degree $n$.
Furthermore, since $V_\k$ is a product form, it follows from
\cite[Prop. 5.6.10]{DX} that
\begin{align*}
&F_n(x) = a_n (x_1^2+x_2^2)^{n/2} \\
  & \quad \times    \left[\frac{n+2 \k_2 + \d_n}{2\k_2+2\k_1}
       C_n^{(\k_2,\k_1)}\left(\frac{x_1}{\sqrt{x_1^2+x_2^2}}\right)
   +   i x_2 C_{n-1}^{(\k_2+1,\k_1)}\left(\frac{x_1}{\sqrt{x_1^2+x_2^2}}\right) \right],
\end{align*}
where $a_n$ is a constant given explicitly in \cite{DX} and $\d_n = 2 \k_2$ if
$n$ is even and $\d_n = 0$ if $n$ is odd. This explicit formula allows us to
compute the norm $\|F_n\|_{\k,p}$ explicitly. For example, using
\cite[Lemma 3.8.9]{DX}, we have immediately that
$$
  \int_{S^d} |F_n(x)|^q h_\k^2(x) d\o(x) = \int_{B^2} |F_n(x)|^q |x_1|^{2\k_1}
     |x_2|^{2\k_2} (1-\|x\|^2)^{|\k| - \k_1 - \k_2 + \f{d-3}{2}} dx _1dx_2,
$$
where, since $F_n$  depends only on $(x_1,x_2)$, we have abused
notation somewhat by using $F_n(x)$ in the right hand side as
well. The above integral can then be evaluated, by \eqref{G-Gegen}
and in polar coordinates, by using \eqref{JacobiLp}. The result,
however, does not yield the sharpness of (ii) in Theorem
\ref{thm:proj} when $\k \ne 0$.
\end{rem}


\section{Boundedness of Ces\`aro means}
\setcounter{equation}{0}

In this section we prove Theorems \ref{thm:Cesaro} and
\ref{thm:Cesaro2}. By the standard duality argument, it suffices
to prove these theorems for $1\leq p \leq 2$.  We shall assume
$1\leq p \leq \nu:= \f{2\sa_\k+2}{\sa_\k+2}$ and $\da>\da_\k(p)$
for the rest of this section.

\subsection{Proof of Theorem \ref{thm:Cesaro}}

We follow essentially the approach of \cite{So2}, although there are still
several difficulties that need to be overcame.

\subsubsection{Decomposition}
Let $\vi_0\in C^\infty [0,\infty)$ be such that $\chi_{[0, 1]}\leq
\vi_0\leq \chi_{[0,2]}$, and let $\vi(t): =\vi_0(t)-\vi_0(2t)$.
Clearly,  $\vi$ is a $C^\infty$-function supported in $(\f12,2)$
and satisfying $\sum_{v=0}^\infty \vi(2^vt)=1$ for all $t>0.$ Set
\begin{equation}\label{5-1}
  \wh{S}_{n,v}^\da (j):= \vi\Bl(\f{2^v(n-j)}n\Br) \f{A_{n-j}^\da}{A_n^\da},
\end{equation}
we define
$$
 S_{n,v}^\da f: = \sum_{j=0}^n  \wh{S}_{n,v}^\da (j) \proj_j(h_\k^2;f), \qquad
       v=0,1,\cdots, \lfloor \log_2 n \rfloor +2.
$$
Since $ \sum_{v=0}^{ \lfloor \log_2 n \rfloor +2} \vi\(\f{2^v(n-j)}n\)=1$ for
$ 0\leq j\leq n-1$,  it follows that the Ces\`aro means are decomposed as
\begin{equation} \label{5-2}
S_n^\da(h_\k^2; f) =\sum_{v=0}^{ \lfloor \log_2 n \rfloor +2} S_{n,v}^\da f + \f
        1{A_n^\da} \proj_n (h_\k^2;f).
\end{equation}

Using Theorem \ref{thm:proj} and the fact that $\da > \da_\k (p)$,
we have
$$
   \f 1{A_n^\da} \|\proj_n (h_\k^2; f)\|_{\k,p}\leq c n^{-\da}
      \|\proj_n(h_\k^2; f)\|_{\k,2}\leq c n^{\da_\k(p)-\da}\|f\|_{\k,p}\leq
       c\|f\|_{\k,p}.
$$
On the other hand, using summation by parts $\ell \ge 1$ times shows that
$$
 S_{n,v}^\d f  = \sum_{j=0}^n  \Delta^\ell \(\wh S_{n,v}^\da(j)\) A_j^{\ell -1}
         S_j^{\ell-1}(h_\k^2;f),
$$
where $\Delta$ denotes the forward difference and $\Delta^{\ell+1}
: = \Delta \Delta^\ell$.  Since $\wh{S}_{n,v}^\da(j)=0$ whenever
$n-j>\f {n}{2^{v-1}}$ or $n-j <\f n{2^{v+1}}$, it is easy to
verify by the Lebnitz rule that
\begin{equation}\label{eq:DeltaSnv}
   \Bl|\Delta^\ell (S_{n,v}^\da(j))\Br|\leq c  2^{-v\da}\(\f {2^v}n\)^\ell,   \qquad
        \forall \ell \in\mathbb{N}, \  0\leq j\leq n.
\end{equation}
Hence, choosing $\ell > \l_\k$ and using the fact that $S_n^\ell (h_\k^2; f)$ is
bounded in $L^p(h_\k^2;S^d)$ for all $1 \le p \le \infty$ if $\ell > \l_\k$, we
conclude that for $v = 0$ and $1$,
$$
    \|S_{n,v}^\da f\|_{\k,p}\leq c n^{-\ell} \sum_{j=0}^n
    j^{\ell-1}  \|S_j^{\ell-1}
             (h_\k^2;f) \|_{\k,p} \le  c \|f\|_{\k, p}.
$$
Therefore, by (\ref{5-2}), it is sufficient to prove that
\begin{equation}\label{5-3}
\|S_{n,v}^\da (f)\|_{\k,p}\leq c 2^{-v\va_0} \|f\|_{\k,p},\ \
     v= 2,\cdots, \lfloor \log_2 n \rfloor +2,
\end{equation}
where $\va_0$ is a sufficiently small positive constant depending on
$\da$ and  $p$, but independent of $n$ and $v$.

\subsubsection{Estimate of  the kernel of $S_{n,v}^\da$}
Let $$ D^\da_{n,v}(t):=\sum_{j=0}^n \wh{S}_{n,v}^\da (j)
\f{\l_\k+j}{\l_\k} C_j^{\l_\k}(t).$$ The definition shows that
$S_{n,v}^\da f = f \ast_\k D_{n,v}^\da$, so that the kernel of
$S_{n,v}^\da f$ is defined by
$$
K^\da_{n,v}(x, y):=V_\k \Bl[ D^\da_{n,v} (\la x, \cdot\ra)\Br](y).
$$

\begin{lem}\label{lem-5-1}
Let $2 \leq v\leq \lfloor \log_2 n\rfloor +2$. Then for any given positive
integer $\ell$,
$$
|K^\da_{n,v}(x,y)| h_\k^2(y)\leq c n^{d} 2^{v(\ell-1-\da)}
    \(1+nd(\bar{x},\bar{y})\)^{-\ell-d+\l_\k+1},
$$ where $\bar{z}=(|z_1|,\cdots, |z_{d+1}|)$ for $z=(z_1,\cdots,
z_{d+1})\in \mathbb{R}^{d+1}$.
\end{lem}

\begin{proof}
We first define a sequence of functions  $\{a_{n,v,\ell}(\cdot)\}_{\ell=0}^\infty$
by
\begin{align*}
a_{n,v,0}(j)&= 2(j +\l_\k) \wh{S}_{n,v}^\da (j),\\
a_{n,v, \ell+1}(j)& =\f{ a_{n,v,\ell}(j)}{ 2j+2\l_\k+\ell}-\f{
a_{n,v,\ell}(j+1)}{ 2j+2\l_\k+\ell+2},   \qquad  \ell\ge 0.
\end{align*}
Following the proof of Lemma 3.3 of \cite[p.413--414]{BD}, we can write,
for any integer  $\ell\ge 0$,
\begin{equation} D^\da_{n,v}(t) =c_\k
  \sum_{j=0}^\infty a_{n,v,\ell}(j) \f{ \Gamma(j+2\l_\k+\ell)}{\Gamma(j+\l_\k+\f12)}
       P_j^{(\l_\k+\ell-\f12, \l_\k-\f12)} (t),
\end{equation}
so that
$$
K_{n,v}^\da(x,y) = c_\k \sum_{j=0}^\infty a_{n,v,\ell}(j) \f{ \Gamma(j+2\l_\k+\ell)}
{\Gamma(j+\l_\k+\f12)} V_\k\Bl[ P_j^{(\l_\k+\ell-\f12, \l_\k-\f12)}(\la x,\cdot\ra)\Br](y).
$$
Note that $a_{n,v,\ell}(j) =0$ if $j +\ell \le (1-\f1{2^{v-1}})n$ or
$j \ge (1-\f1{2^{v+1}})n$, so that the sum is over $j \sim n$.
Furthermore, it follows
from the definition, \eqref{eq:DeltaSnv} and Lebinitz rule that
\begin{align} \label{5-5}
 \Bl|\triangle^i a_{n,v,\ell}(j)\Br|&\leq c 2^{-v\da} n^{-\ell+1}
\(\f{2^v}{n}\)^{i+\ell},\qquad  i,\ell=0,1,\cdots .
\end{align}
Consequently, using the pointwise estimate of \eqref{VJacobi},
it follows  by (\ref{5-5}) that
\begin{align*}
  |K_{n,v}^\da(x,y)|&\leq cn^{2\l_\k+2\ell-1-2|\k|}\sum_{\substack{
  j\sim n\\
   n-j\sim \f n{2^v}}}
     | a_{n,v,\ell}(j)| \f{ \prod_{i=1}^{d+1}(|x_iy_i|+n^{-1}
    d(\bar{x}, \bar{y})+n^{-2})^{-\k_i}}
     {(1+nd(\bar{x}, \bar{y}))^{\l_\k+\ell-|\k|}}\\
&  \leq c n^{d} 2^{v(\ell-1-\da)}\f{\prod_{j=1}^{d+1}(|x_jy_j|+n^{-1} d(\bar{x},
      \bar{y})+n^{-2})^{-\k_j}}{(1+nd(\bar{x}, \bar{y}))^{\l_\k+\ell-|\k|}}\\
& \leq c n^{d} 2^{v(\ell-1-\da)} h_\k^{-2}(y)\(1+nd(\bar{x},\bar{y})\)^{\l_\k -d+1-
\ell},
\end{align*}
where in the last inequality we have used the fact that
$$
\prod_{j=1}^{d+1}(|x_jy_j|+n^{-1}d(\bar x, \bar y)+n^{-2})^{-\k_j}
\le c h^{-2}(y) d(\bar x,\bar y)^{|\k|},
$$
which follows since if $|y_j|\ge 2d(\bar{x},\bar{y})$, then $|\bar x_j - \bar y_j|
\le d(\bar x, \bar y) \le |y_j|/2$ so that $|y_j|^2 \le  2 |x_j y_j|$, whereas if
$|y_j|< 2d(\bar{x},\bar{y})$ then $|y_j|^2  \le 2 (n^{-1} d(\bar{x},\bar{y}) )
   \cdot n d(\bar{x},\bar{y})$. This completes the proof of Lemma \ref{lem-5-1}.
\end{proof}

\begin{cor}\label{cor-5-2}
For any $\ga>0$ there exists an $\va_0>0$ independent of $n$
and $v$ such that
$$
\sup_{x\in S^d} \int_{\{y:\  d(\bar{x},\bar{y})>2^{(1+\ga)v}/n \}}
   |K_{n,v}^\da (x,y)|\, h_\k^2(y)\, d\og(y) \leq c 2^{-v\va_0}.
$$
\end{cor}

\begin{proof}
Invoking Lemma \ref{lem-5-1} with $\ell> \l_\k+1+\f{\l_\k-\da}\ga$, we see that
the quantity to be estimated is bounded by
\begin{align*}
&  c \sup_{x\in S^d} n^d 2^{v(\ell-1-\da)}
   \int_{\{y:\  d(\bar{x},\bar{y})>2^{(1+\ga)v}/n \}}
          \frac{1}{(1+n d (\bar x, \bar y))^{\ell + d-\l_k-1}} d\og(y) \\
& \qquad \le c 2^{v(\ell-1-\d)} \int_{2^{(1+\ga)v}/n}^\pi
        \f {n (n\t)^{d-1}}{(1+n \t)^{\ell + d-\l_k-1}} d\t  \\
 &\qquad \le c 2^{v\(\ell-1- \da-(1+\ga)(\ell-\l_\k-1)\)} = c 2^{-v \va_0}
\end{align*}
which proves the corollary.
\end{proof}

\subsubsection{ Proof of \eqref{5-3} } Now we are in a position to prove
(\ref{5-3}). Recall that
\begin{align}
 S_{n,v}^\da f =  \sum_{(1-2^{-v+1})n\leq j\leq (1-2^{-v-1})n}
   \wh{S}_{n,v}^\da (j) \proj_j(h_\k^2;f). \label{5-7}
\end{align}

Assume $\da>\da_\k(p)$, and let  $\ga>0$ be  sufficiently small so
that $\da> \da_\k(p)+\ga \(\da_\k(p)+\f12\)$. Set $v_1=v (1+\ga)$.
Let $\Lambda$ be a  maximal $\f{2^{v_1}}n$-separable subset of
$S^d$; that is, $\min_{\varpi \neq \varpi' \in\Lambda}
d(\varpi,\varpi')\ge \f{2^{v_1}}n$ and $S^d \subset
\cup_{\varpi\in\Lambda} c(\varpi, \f{2^{v_1}}n)$. Define
$$
 f_\varpi(x): = f(x) \chi_{c(\varpi,\f{2^{v_1}}n)}(x) [A(x)]^{-1}, \qquad
    A(x):=  \sum_{\varpi \in \Lambda} \chi_{c(\varpi,\f{2^{v_1}}n)}(x).
$$
Then evidently $1 \le A(x) \le c$, $x \in S^d$, $|f_\varpi|\leq c |f|$,
and $f(x)= \sum_{\varpi\in\Lambda} f_\varpi(x)$. Using the Minkowski inequality,
we obtain
$$
 \|S_{n,v}^\da(f)\|_{\k,p}\leq \sum_{\varpi\in\Lambda} \|S_{n,v}^\da (f_\varpi)\|_{\k,p}.
$$
Thus, it is sufficient to show that for each $\varpi\in\Lambda$, we have
\begin{equation}\label{5-9}
       \|S_{n,v}^\da (f_\varpi)\|_{\k,p}\leq c 2^{-v\va_0}\|f_\varpi\|_{\k,p}.
\end{equation}

To this end, we denote by $c^*(\varpi, 2^{v_1+1}/n)$ the set
$$
 c^*(\varpi, \tfrac{2^{v_1+1}}n)=\left\{x\in S^d:\  d(\bar{x},
    \bar{\varpi})\leq 2^{v_1+1}/n\right\}
$$
and further define $J(v,n):= \{j: (1-2^{-v+1})n\leq j\leq
(1-2^{-v-1})n\}$. Using \eqref{5-7} and orthogonality, we  obtain
\begin{align*}
 \|S_{n,v}^\da (f_\varpi) \|_{\k,2} = \Bl(\sum_{j \in J(v,n)}
  |\wh{S}_{n,v}^\da (j)|^2 \|\proj_j(h_\k^2;f_\varpi)\|_{\k,2}^2\Br)^{\f12}.
\end{align*}
Hence, by H\"older's inequality, Theorem \ref{thm:proj-cap} and
\eqref{eq:doubling}, and (\ref{eq:DeltaSnv}) with $\ell=0$,
\begin{align*}
& \Bl(\int_{c^*(\varpi,  2^{v_1+1}/n)} |S_{n,v}^\da (f_\varpi) (x)|^p \,
      h_\k^2(x)\, d\og(x)\Br)^{\f1p} \\
& \quad \leq c \Br(\int_{c(\varpi, 2^{v_1+1}/n)} h_\k^2(x)\,
  d \og(x)\Br)^{\f1p-\f12}\Bl(\sum_{j \in J(v,n)}
  |\wh{S}_{n,v}^\da (j)|^2 \|\proj_j(h_\k^2;f_\varpi)\|_{\k,2}^2\Br)^{\f12}\\
& \quad \leq c  2^{v_1(\da_\k(p)+\f12)} n^{-\f12}  \Bl(\sum_{j \in
J(v,n)}
  |\wh{S}_{n,v}^\da (j)|^2 \Br)^{\f12}\|f_\varpi\|_{\k,p}\\
& \quad \leq
c2^{-v\(\da-\da_\k(p)-\ga(\da_\k(p)+\f12)\)}\|f_\varpi\|_{\k,p}
  = c2^{-v\va_0}\|f_\varpi\|_{\k,p}.
\end{align*}

Finally, using H\"older's inequality, we obtain, for  $x \notin
c^*(\varpi, 2^{v_1+1}/n)$,
\begin{align*}
  |S_{n,v}^\da (f_\varpi)(x)|^p & =\Bl|\int_{\{y: d(\varpi, y)\le
    2^{v_1}/n\}} f_\varpi(y) K_{n,v}^\da(x,y) h_\k^2(y)\, d\og(y)\Br|^p\\
&\le \Bl(\int_{\{y:\ d(\bar{y}, \bar{x})\ge 2^{v_1}/n\}}
     |f_\varpi(y)|^p |K_{n,v}^\da(x,y)| h_\k^2(y)\,d\og(y)\Br)\\
&\qquad \times \Bl(\int_{\{y:\ d(\bar{y},\bar{x})\ge 2^{v_1}/n\}} |
     K_{n,v}^\da(x,y)| h_\k^2(y)\, d\og(y)\Br)^{p-1},
\end{align*}
which, together with Corollary \ref{cor-5-2}, implies
\begin{align*}
& \Bl(\int_{S^d\setminus c^*(\varpi, 2^{v_1+1}/n)} |S_{n,v}^\da (f_\varpi)(x)|^p\,
  h_\k^2(x)\, d\og(x)\Br)^{\f1p}\\
&\quad \leq c(2^{-v \va_0})^{1-\f1p} \sup_{y\in S^d}\Bl(\int_{\{x:
\ d(\bar{x}, \bar{y})\ge
     2^{v_1}/n\}} | K_{n,v}^\da(x,y)| h_\k^2(x)\,d\og(x)\Br)^{\f1p}
          \|f_\varpi\|_{\k,p}\\
&\quad \leq c 2^{-v\va_0}\|f_\varpi\|_{\k,p}.
\end{align*}

Putting the above together, we deduce  the desired estimate
(\ref{5-9}), hence \eqref{5-3}, and complete the proof of Theorem
\ref{thm:Cesaro}. \qed

\subsection{Proof of Theorem \ref{thm:Cesaro2}}

\subsubsection{Main body of the proof}
For the proof of Theorem \ref{thm:Cesaro} we follow the approach
in \cite{AH}, which can be traced back to \cite{NR}. We start with the
following lemma.

\begin{lem} \label{lem:QQ}
If $Q$ is a polynomial of degree $n$ on $\RR^{d+1}$, then for
$1\le p < \infty$,
$$
   \|Q\|_\infty:=\max_{x\in S^d}|Q(x)| \le c n^{(2\s_\k +1)/p} \|Q\|_{\k,p}.
$$
\end{lem}

\begin{proof}
Let $S_n(h_\k^2;f)$ denote the partial sum operator of the $h$-harmonic
expansion. Then $S_n(h_\k^2;Q) = Q$. The kernel of $S_n(h_\k^2;f)$ is
$$
 K_n(h_\k^2; x,y) = \sum_{k=0}^n P_k(h_\k^2; x,y), \qquad
    P_k(h_\k^2;x,y) = \sum_j Y_j^k(x) Y_j^k(y),
$$
where $\{Y_{j}^k\}_j$ forms an orthonormal basis of
$\CH_k^{d+1}(h_\k^2)$. If $y \in S^d$ and $|y_\ell| = \max_{1 \le
j \le d+1} y_j$, then $|y_\ell| \ge 1/\sqrt{d+1}$. Hence, the
pointwise estimate of \eqref{eq:S-est} implies that
$|P_n(h_\k^2;x,x)| \le c n^{2 \s_\k}$. By the Cauchy-Schwarz
inequality, we obtain
$$
 |K_n(h_\k^2;x,y)| \le \sum_{k=0}^n P_k(h_\k^2;x,x)^{\f12}
  P_k(h_\k^2;y,y)^{\f12} \le c \sum_{k=0}^n k^{2 \s_\k} \le c n^{2 \s_\k+1}.
$$
Consequently, we conclude that
$$
  \|Q\|_\infty =   \|S_n(h_\k^2;Q)\|_\infty =
   \left \| a_\k \int_{S^d} Q(y) K_n(h_\k^2;x,y)
         h_\k^2(y)d\o  \right \|_\infty \le c n^{2\s_\k+1} \|Q\|_{\k,1}.
$$
Furthermore, we clearly have $\|Q\|_\infty  \le \|Q\|_\infty$,  and the case
$1 < p < \infty$ follows immediately from interpolation of these two cases.
\end{proof}

\medskip\noindent
{\it Proof of Theorem \ref{thm:Cesaro2}}. Our main objective is to show that
\begin{equation} \label{eq:6.9}
       \sup_{n\in \NN}  \|S_n^\d(h_\k^2; f)\|_{\k,p} \le c \|f\|_{k,p}
\end{equation}
does not hold if $1 \le  p \le \frac{2\s_\k +1}{\s_\k + \d +1}$ or
$p \ge \frac{2\s_\k +1}{\s_\k-\da}$. Let
$$
 p_1:=\frac{2\s_\k +1} {\s_\k-\da} \qquad \hbox{and} \qquad
     q_1 := \frac{p_1}{p_1-1} =  \frac{2 \s_\k +1}{\s_k+1+\d}.
$$
It is sufficient to prove that \eqref{eq:6.9} does not hold for $p_1$,
since it then follows from the Riesz-Thorin convexity theorem that
\eqref{eq:6.9} fails for $p_1 \le p \le \infty$ and that \eqref{eq:6.9} fails
for $1 \le p \le q_1$ follows by duality.

Let $e \in S^d$ be fixed. Define a linear functional $T_n^\da:
L^p \mapsto \RR$ by
$$
  T_n^\da f : = S_n^\da (h_\k^2;f, e) = a_\k \int_{S^d} f(x)
           K_n^\d(h_\k^2; x,e) h_\k^2(x) d\o(x).
$$
Since this is an integral operator, a standard argument shows that
\begin{equation*}
  \|T_n^\d\|_{\k,p} = \|K_n^\d(h_\k^2;x,e)\|_{\k,q}, \qquad \f 1 p + \f 1 q
  =1,
\end{equation*}
where $ \|T_n^\d\|_{\k,p} =\sup_{\|f\|_{\k,p}=1}|T_n^\d f|$.  On
the other hand, by Lemma \ref{lem:QQ}, if \eqref{eq:6.9} holds,
then we will have
\begin{align*}
   |T_n^\d f|& = |S_n^\d (h_\k^2;f,e)| \le \|S_n^\d (h_\k^2;f)\|_{\k,\infty} \\
   &  \le c n^{(2 \s_\k+1)/p} \|S_n^\d (h_\k^2; f)\|_{\k,p}
         \le c n^{(2 \s_\k+1)/p} \| f\|_{\k,p}.
\end{align*}
Consequently, the above two equations show that we will have
\begin{equation}\label{eq:6.10}
  \|K_n^\d (h_\k^2;\cdot,e)\|_{\k,q} \le c n^{ (2 \s_\k+1) /p}, \qquad \f 1 p + \f 1 q =1.
\end{equation}
To complete the proof of  Theorem \ref{thm:Cesaro2}, we show that
\eqref{eq:6.10} does not hold for $p = p_1$.

For this purpose we use the explicit formula for $K_n^\d(h_\k^2;x,e)$ for
$e = e_j$  (see, for example, \cite{LX}), where $e_1,\ldots, e_{d+1}$
denote the usual coordinate vectors in $\RR^{d+1}$,
$$
    K_n^\da (h_\k^2; x,e_j) : = K_n^\d (w_{\l_\k - \k_j, \k_j}; 1, x_j),
$$
where $K^\da_n(w_{\l,\mu}; s,t)$ denotes the $(C,\d)$ kernel of
the generalized Gegenbauer polynomials with respect to the weight
function \eqref{GGweight}. Hence, we have
$$
   \int_{S^d} \left |K_n^\d (h_\k^2; x,e_j)\right|^q h_\k^2(x) d\o(x) =
      c \int_{-1}^1 \left|K_n^\d (w_{\l_\k-\k_j,\k_j}; 1, t)\right|^q
          w_{\l_\k-\k_j,\k_j}(t) dt.
$$
Consequently, choosing $j$ such that $\k_j = \k_{\rm min}$, we see
that the proof of Theorem \ref{thm:Cesaro2} follows by applying
Proposition \ref{prop:6.4} below to $\sa=\sa_\k$ and $\mu=\k_{\rm
min}$. \qed

\begin{prop}\label{prop:6.4}
Let $w_{\sa,\mu}$ be the weight function in \eqref{GGweight} and
$\sa \ge \mu \ge 0$. Define
$$
  \Phi_{n,q}^\d(w_{\sa,\mu}, s): = \int_{-1}^1 |K_n^\d (w_{\sa,\mu}; s, t)|^q
     w_{\sa,\mu}(t) dt.
$$
Then for $q_1 = \frac{2 \sa +1}{\sa+1+\d}$ and $p_1 =
\frac{2\sa+1}{\sa-\d}$,
$$
  \Phi_{n,q_1}^\d(w_{\sa,\mu}, 1) \ge   \Phi_{n,q_1}^\d(w_{\mu, \sa}, 0) \ge
         c n^{(2\sa +1)q_1/p_1} \log n.
$$
\end{prop}

The proof of this proposition is given in the following subsection.

\subsubsection{Proof of Proposition \ref{prop:6.4}}
The case of $q =1$ and $\d = \sa$ has already been established in
\cite{DaiX, LX}. We follow the approach in \cite{DaiX} and briefly
sketch the proof.

Using the sufficient part of Corollary \ref{thm:CesaroGG}, we can follow
exactly the deduction in \cite[p. 293]{LX} to obtain
\begin{align*}
 & \Phi_{n,q_1}^\d(w_{\sa,\mu}; 1) \ge  \Phi_{n,q_1}^\d(w_{\mu,\sa}; 0) =
  c n^{(\sa+\mu-\d+ \frac{1}{2})q_1} \\  & \quad \times \int_{0}^1 \left |
   \int_{-1}^1 P_n^{(\sa+ \mu +\d+\frac12, \sa+\mu -\frac12)} (s t)
   (1-s^2)^{\mu-1} ds\right|^{q_1} t^{2\mu}(1-t^2)^{\sa -\frac12} dt + \CO(1).
\end{align*}
As  a result, we see that Proposition \ref{prop:6.4} is a consequence of the
lower bound of a double integral of the Jacobi polynomial given in the next
proposition.

\begin{prop} \label{prop:lower}
Assume $ \sa, \mu\ge 0$ and $0 \le \da \le  \sa + \mu$. Let $a =
\sa + \mu + \d $, $b = \sa+\mu -1$ and $q_1 =
\frac{2\sa+1}{\sa+1+\da}$. Then
\begin{align} \label{lowerJacobi}
&\int_{0}^1 \left| \int_{-1}^1 P_n^{(a+\frac12,b+\frac12)} (st)
(1-s^2)^{\mu-1} ds \right|^{q_1} t^{2\mu} (1-t^2)^{\sa -1/2} dt \\
 & \qquad\qquad\qquad \ge c \, n^{-(\mu +1/2)q_1}  \log n.  \notag
\end{align}
\end{prop}

\begin{proof}
Denote the left hand side of \eqref{lowerJacobi} by $I_{n,q_1}$.
First assume that $0 < \mu < 1$. Following the proof of
\cite{DaiX}, we can conclude that
$$
  I_{n,q_1} \ge c n^{-q_1/2} \int_{n^{-1}}^{\pi/4} |M_n(\phi)|^{q_1}
    (\sin \phi)^{2\sa} d\phi - \CO(1) E_{n,q_1},
$$
where
$$
   E_{n,q_1} : = n^{-\frac32 q_1 }   \int_{ n^{-1}}^{\pi/4}
     \[ \int_\phi^{\pi-\phi}  \frac{(\cos^2 \phi-\cos^2 \t)^{\mu-1} }
        { (\sin \frac{\t}2 )^{a+1} (\cos \frac{\t}2)^{b+1} } d\t \]^{q_1}
          \, (\sin \phi)^{2\sa}  d\phi
$$
and $M_n(\phi) := K_n(\phi) + G_n(\phi)$ satisfies
$$
  K_n(\phi) \ge c n^{-\mu} \phi^{-\sa - \d -1} \(1 + \cos (2N\phi + 2 \g) \), \qquad
       n^{-1} \le \phi\le \va,
$$
for a sufficiently small absolute constant $\va > 0$, where
$N=n+\f{a+b}2+1$ and $\g = -\f\pi2(a+1 -\mu)$, and
$$
|G_n(\phi)| \le c n^{-1} \phi^{\mu-\sa - \d -2}, \qquad n^{-1}
\leq \phi < \pi/4.
$$

From these estimates and the fact that $q_1 (\sa+1+\da) = 2\sa+1$,
it follows that
\begin{align*}
 & \int_{n^{-1}}^\va |K_n(\phi)|^{q_1} (\sin \phi)^{2\sa}  d\phi \ge
    c n^{-\mu q_1} \int_{n^{-1}}^\va \phi^{2\sa-q_1 (\sa+1+\da)}
    (1+ \cos (2 N\phi +2\g))^{q_1} d\phi \\
 & =  c n^{-\mu q_1} \int_{n^{-1}}^\va \phi^{-1} (1+ \cos (2 N\phi +2\g))^{q_1} d\phi
      \ge c n^{-\mu q_1} \log n,
\end{align*}
where in the last step we used $(1+ A)^{q_1} \ge 1 + q_1 A$ for
$A\in [-1,1]$ and the fact that  $\int_{n^{-1}}^\va \phi^{-1}
 \cos  (2 N\phi +2\g))d\phi \le c $
upon using integral by part once. Furthermore,
\begin{align*}
   \int_{n^{-1}}^\va |G_n(\phi)|^{q_1} (\sin \phi)^{2\sa} d \phi
  \le\ & c n^{-q_1} \int_{n^{-1}}^\va \phi^{2\sa + {q_1}(\mu-\sa-\d-2)} d \phi  \\
   =  \ &  c n^{-q_1} \int_{n^{-1}}^\va \phi^{q_1(\mu-1)-1} d \phi  \le c n^{-q_1\mu}.
\end{align*}
Together, these estimates yield that for $0 < \mu < 1$,
$$
 \int_{n^{-1}}^{\pi/4} |M_n(\phi)|^{q_1} (\sin \phi)^{2\sa} d\phi \ge
       c n^{-q_1 \mu} \log n.
$$

Moreover, the remainder $E_{n,q_1}$ term can be estimated as
follows
\begin{align*}
 E_{n,q_1} & \le c n^{-\f32 q_1} \int_{n^{-1}}^{\pi/4} \[ \int_{\phi}^{\pi/4}
      \t^{\mu-a-2} (\t - \phi)^{\mu-1} d\t  \]^{q_1} \phi^{2\sa} d\phi \\
     & \le c n^{-\f32 q_1} \int_{n^{-1}}^{\pi/4} \phi^{(\mu-\sa-\d-2){q_1}+2 \sa} d\phi \\
     &  \le c n^{-\f32 q_1 - q_1 (\mu-1)} = c n^{-(\mu+\f12)q_1},
\end{align*}
where in the second step, we divided the inner integral as two parts,
over $[\phi, 2 \phi]$ and over $[2\phi, \pi/2]$, respectively, to derive
the stated estimate.

Putting these two terms together, we conclude the proof for the case
$0 < \mu < 1$. The case $\mu = 1$ can be derived similarly upon taking
an integration by parts for the inner integral in \eqref{lowerJacobi}. The
case $\mu > 1$ reduces to the case $0 < \mu < 1$ upon taking integration
by parts $\lfloor \mu \rfloor$ times as in \cite{DaiX}.
\end{proof}


\section{Proof of theorems on the ball and on the simplex}
\setcounter{equation}{0}

\subsection{Proof of results on the unit ball}

Under the mapping $\phi: x\in B^d \mapsto (x, \sqrt{1-\|x\|^2}) \in S^d_+$
in \eqref{B-S}, orthogonal polynomials with respect to $W_\k^B$ in
\eqref{weightB} can be deduced from $h$-spherical harmonics that are
even in $x_{d+1}$. Moreover, the connection \eqref{projBS} shows that
the $(C,\d)$ means $S_n^\d (W_\k^B; f)$ is related to
$S_n^\d(h_\k^2; F)$ by
\begin{equation*} 
  S_n^\d(W_\k^B; f,x ) = S_n^\d (h_\k^2; F, X), \quad X:= (x,\sqrt{1-\|x\|^2}),
 \end{equation*}
where $F(x, x_{d+1}) : = f(x)$ for $x \in B^d$ and $(x,x_{d+1}) \in S^d$.
Consequently, by \eqref{BSintegral}, Theorem \ref{thm:CesaroBT} with
$\Omega = B$ follows immediately from Theorem \ref{thm:Cesaro}.

The proof of Thereom \ref{thm:CesaroBT2} follows almost exactly as that
of Theroem \ref{thm:Cesaro2}. We have in this case (\cite[p. 287-288]{LX})
\begin{align*}
K_n (W_\k^B; x, e_j) & \ = K_n^\d (w_{\l_\k - \k_j, \k_j}; 1, x_j), \qquad 1 \le j \le d, \\
K_n (W_\k^B; x, 0)    & \ = K_n^\d (w_{k_{d+1}, \l_\k - \k_{d+1}}; \|x\|, 0).
\end{align*}
Hence, by \eqref{BSintegral}, we can again reduce the proof of
Theroem \ref{thm:CesaroBT2} to the lower bound of
$\Phi_{n,q}^\d(w_{\sa,\mu},1)$ and $\Phi_{n,q}^\d(w_{\mu,\sa},0)$,
which follows again from Proposition \ref{prop:6.4}.

\subsection{Proof of results on the simplex}

\subsubsection{Projection operator}
Under the mapping $\psi: x \in T^d \mapsto (x_1^2,\ldots,x_d^2)
\in B^d$, orthogonal polynomials with respect to $W_\k^T$ on $T^d$
and those with respect to $W_\k^B$ on $B^d$ are related. In
particular, we have the connection between $\proj_n(W_\k^T;f)$ and
$\proj_{2n}(W_\k^B; f\circ \psi)$ given in \eqref{projTB}, from
which the result on the projection operators can be readily
derived.

In fact, if $\|\proj_n (W_\k^B; f)\|_{W_k^B,p} \le A_n
\|f\|_{W_\k^B;p}$, then by \eqref{projTB} and \eqref{T-B},
\begin{align*}
   \|\proj_n(W_\k^T; f)\|_{W_\k^T;p}
   &  = \|\proj_n(W_\k^T; f)\circ \psi\|_{W_\k^B;p} \\
    & = \frac{1}{2^d} \Bl \| \sum_{\va \in \ZZ_2^d}
           \proj_{2n} (W_\k^B; f\circ \psi, \cdot\va)  \Br \|_{W_\k^B;p}    \\
    &    \le \|\proj_{2n}(W_\k^B; f\circ \psi) \|_{W_\k^B;p} \\
    & \le   A_{2n}  \|f\circ \psi\|_{W_\k^B;p} = A_{2n} \|f\|_{W_\k^T;p},
\end{align*}
from which Theorem \ref{thm:projBT} for $\Omega = T$ follows immediately
from the case $\Omega = B$. Furthermore, since the distance $d_T(x,y)$ on
$T^d$ is related to the geodesic distance on $S^d$ by
$$
  d_T\(\psi(x),\psi(y)\) = d(X,Y), \qquad X = \(x,\sqrt{1-\|x\|^2}\), \quad
           Y = \(y,\sqrt{1-\|y\|^2}\),
$$
from \eqref{BSintegral} and \eqref{T-B} it follows readily that
$$
   \int_{c_T(x,\t)}W_\k^T(x) d x = \int_{c(X,\t)} h_\k^2(y) d\o(y),$$
   where
         $X = \(\sqrt{x_1},\cdots, \sqrt{x_d}, \sqrt{1-|x|} \).$
Consequently, we conclude that Theorem \ref{thm:proj-capT} follows from
Theorem \ref{thm:proj-cap}.

\subsubsection{Ces\`aro means}
Since the connection \eqref{projTB} relates the
projection operator of degree $n$ for $W_\k^T$ to the projection
operator of degree $2n$ for $W_\k^B$, we cannot deduce results for
the Ces\`aro means $S_n^\d(W_\k^T;f)$ from those of $S_n^\d(W_\k^B;f)$
directly. We can, however, follow the proof of the theorems for the
$h$-harmonics. Below we give a brief outline on how this will work out.

\medskip\noindent
{\it Proof of Theorem \ref{thm:CesaroBT} with $\Omega = T$.}
We follow the decomposition in the subsection 6.1.1  to define
$$
 S_{n,v}^\d(W_\k^T; f) = \sum_{j=0}^n \wh S_{n,v}^\d(j)\proj_j(W_\k^T;f),
    \qquad v = 1,2,\ldots,  \lfloor \log_2 n \rfloor +2.
$$
The same argument shows that it suffices to prove the analogue
of \eqref{5-3},
\begin{equation} \label{eq:7.1}
  \|S_{n,v}^\d (W_\k^T;f)\|_{W_\k^T,p} \le c 2^{-v \va}\|f\|_{W_\k^T;p},
         \qquad v = 2,\ldots,  \lfloor \log_2 n \rfloor +2.
\end{equation}
Denote the kernel of $S_{n,v}^\d(W_\k^T;f)$ by $K_{n,v}^\d(W_\k^T;x,y)$.
Then we have by \eqref{proj-kernelT} that
$$
    K_{n,v}^\d(W_\k^T;x,y) : = c_\k\int_{[-1,1]^{d+1}}
        D_{n,v}^\d(W_\k^T;2 z(x,y,t)^2 -1) \prod_{i=1}^{d+1} (1-t_i^2)^{\k_i-1}dt
$$
where
$$
   D_{n,v}^\d (W_\k^T; t): = \sum_{j=0}^n \wh{S}_{n,v}^\da(j)
    \frac{(2j+\l_\k)\Gamma(\frac12)
          \Gamma(j+\l_k)}  {\Gamma(\l_\k+1)\Gamma(j+\frac12)} \\
               P_j^{(\l_k-\frac12, \f 12)} (t).
$$
Consequently, define analogues of $a_{n,v,\ell}$ by
\begin{align*}
   a^T_{n,v,0}(j) \ & =  (2j + \l_k) \wh S_{n,v}^\d(j) \\
   a^T_{n,v,\ell+1}(j) \ & = \frac{a_{n,v,\ell}(j)}{2j + 2 \l_\k + \ell} -
          \frac{a_{n,v,\ell}(j+1)}{2j + 2 \l_\k + \ell +2},
\end{align*}
we can then write, again following the proof of Lemma 3.3 of
\cite{BD}, that
$$
 D_{n,v}^\d (W_\k^T; t) = c \sum_{j=0}^n a_{n,v,\ell}^T(j)
      \frac{\Gamma(j+2\l_\k + \ell)}{\Gamma(j+\l_\k+\f12)}
           P_j^{(\l_k+\ell -\frac12, - \f 12)} (t).
$$
The estimate \eqref{5-5} holds exactly for $a_{n,v,\ell}(j)$. Thus, to follow
the proof of Lemma \ref{lem-5-1}, we need to estimate
\begin{equation}\label{eq:intVTP}
 \int_{[-1,1]^{d+1}}
        P_j^{(\l_k+\ell -\frac12,  - \f 12)} (2 z(x,y,t)^2 -1)
           \prod_{i=1}^{d+1} (1-t_i^2)^{\k_i-1}dt.
\end{equation}
Using the fact that $P_n^{(\a,-1/2)}(2 t^2 -1)$ can be written in terms of
$P_n^{(\a,\a)}$ (\cite[(4.1.5)]{Szego}), we can estimate \eqref{eq:intVTP}
again by Lemma \ref{VJacobi}. The result is
$$
 |K_{n,v}^\d(W_\k^T;x,y)| \[W_\k^T(y)\]^{-1}\le
      c n^d 2^{v(\ell-1-\da)} \(1+nd(\bar x, \bar y) \)^{-\l_\k - \ell + d -1},
$$
which implies that the analogue of Corollary \ref{cor-5-2} holds; that is, for any
$\g > 0$ there is an $\va_0 > 0$ such that
$$
  \sup_{x \in T^d} \int_{\{y: d_T(x,y) \ge 2^{(1+v)\g}/n\}}
  |K_{n,v}(W_\k^T;x,y)|
      W_\k^T(y)dy \le c 2^{- v \va_0} .
$$

In order to prove \eqref{eq:7.1}, we then define $\Lambda$ to be a
maximal separate subset of $T^d$ exactly as the one we defined in
 Subsection 6.1.3, except with $d(\varpi,\varpi')$ replaced by
$d_T(x,x')$. Define
$$
    f_y(x): = f(x) \chi_{c_T(y,\f{2^{v_1}}{n})}(x) [A(x)]^{-1}, \qquad
        A(x) = \sum_{y \in \Lambda} \chi_{c_T(y,\f{2^{v_1}}n)}(x).
$$
Then the same argument shows that it suffices to show that
$$
    \|S_{n,v}^\d(W_\k^T; f_y)\|_{W_\k^T,p} \le c 2^{-v \va_0} \|f_y\|_{W_\k^T;p}.
$$
This last inequality can be established exactly as in \eqref{5-9} and there
is no need to introduce the additional set $c^*(\varpi,\t)$.
\qed

\medskip
\noindent
{\it Proof of Theorem \ref{thm:CesaroBT2}.}
We first note that the analogue of Lemma \ref{lem:QQ} holds;
that is, for $1\le p < \infty$,
$$
   \|Q\|_\infty:=\max_{x\in T^d}|Q(x)| \le c n^{(2\s_\k +1)/p} \|Q\|_{W_\k^T,p}
$$
for any polynomial of degree $n$ on $\RR^d$. Furthermore,
following the proof of Theorem \ref{thm:Cesaro2}, it is sufficient
to prove that
\begin{equation}\label{eq:7.3}
 \|K_n^\d(W_\k^T; \cdot,y)\|_{\k,q} \le c n^{ (2 \s_\k+1) /p}, \qquad \tf 1 p + \tf 1 q =1,
\end{equation}
where $y \in T^d$ is fixed, does not hold for $p = p_1: = \frac{2\s_\k+1}{\s_\k-\d}$.
To proceed, we then express $K_n^\d(W_\k^T)$ in terms of the kernel
for the Jacobi polynomial expansions (\cite[p. 290]{LX})
\begin{align*}
 K_n^\d\(W_\k^T;x, e_j\) & =
       K_n^\d \left(w^{(\l_\k -\k_j-\frac12, \k_j -\frac12)};1, 2x_j-1\right),
                        \quad 1 \le j\le d, \\
 K_n^\d(W_\k^T;x, 0) & =
       K_n^\d \left(w^{(\l_\k -\k_j-\frac12, \k_j -\frac12)};1, 1-2|x|\right),
\end{align*}
from which we can deduce by changing variables that
\begin{align*}
   \int_{T^d} & \left|K_n^\d\(W_\k^T; e_j,y\)\right|^q W_\k^T(y) dy  \\
     & = c
     \int_{-1}^1  \left | K_n^\d \(w^{(\l_\k -\k_j-\frac12, \k_j -\frac12)};1,t\)\right|^q
             w^{(\l_\k -\k_j-\frac12, \k_j -\frac12)}(t) dt
\end{align*}
for $1 \le j \le d$, and also for $j=d+1$ if we agree that $e_{d+1} =0$.  As a
result, we have reduced the problem to that of Jacobi polynomial expansions,
so that the desired result follows from \cite{Ch-Mu}.
\qed

\enddocument